\documentclass[12pt,reqno]{amsart}
\usepackage{amsmath,amsthm}
\usepackage{amssymb,latexsym}
\usepackage{graphicx,subfig}
\usepackage[shortlabels]{enumitem}
\usepackage{epsfig}
\usepackage{color}
\usepackage{fullpage}
\usepackage{hyperref}

\usepackage{tikz,epsfig,color,graphicx,epstopdf,subfig}
\usetikzlibrary{positioning,chains,fit,shapes,calc, fit}

\newtheorem{thm}{Theorem}[section]
\newtheorem{lma}[thm]{Lemma}

\newtheorem{prp}[thm]{Proposition}

\newtheorem{qns}[thm]{Question}

\newtheorem{clm}{Claim}
\newtheorem*{clm*}{Claim}

\newcommand{\s}[1]{\mathcal{#1}}

\newcommand{\N}{\mathbb{N}}
\newcommand{\Z}{\mathbb{Z}}
\newcommand{\npart}[2]{N^{#1,#2}}
\newcommand{\floor}[1]{\left\lfloor #1 \right\rfloor}
\newcommand{\ceiling}[1]{\left\lceil #1 \right\rceil}

\def\eps{\varepsilon}

\title{Transitive triangle tilings in oriented graphs}

\author{J\'{o}zsef Balogh}
\address{
Department of Mathematics, University of Illinois, Urbana, IL 61801, USA and Bolyai Institute, University of Szeged, Szeged, Hungary.
    Research is partially supported NSF CAREER Grant DMS-0745185, Arnold O. Beckman Research Award (UIUC Campus Research Board 13039) and Marie Curie FP7-PEOPLE-2012-IIF 327763.
  Research is partially supported by NSF Grant DMS-1500121.} 
\email{jobal@math.uiuc.edu}

\author{Allan Lo}
\address{School of Mathematics, University of Birmingham, Birmingham, B15~2TT, UK.
Research is partially supported by the European Research Council
under the ERC Grant Agreement no. 258345.
}
\email{s.a.lo@bham.ac.uk}

\author{Theodore Molla}
\address{Department of Mathematics, University of Illinois, Urbana, IL 61801, USA.
  Research is partially supported by NSF Grant DMS-1500121.} 
\email{molla@illinois.edu}

\date{\today}

\begin{document}
%\linenumbers

\maketitle

\begin{abstract}
In this paper, we prove an analogue of Corr{\'a}di and Hajnal's classical theorem.
  There exists $n_0$ such that for every $n \in 3\Z$ when
  $n \ge n_0$ the following holds.
  If $G$ is an oriented graph on $n$ vertices and every vertex
  has both indegree and outdegree at least $7n/18$,
  then $G$ contains a perfect transitive triangle tiling,
  which is a collection of vertex-disjoint transitive triangles covering 
  every vertex of~$G$.
  This result is best possible, as, for every $n \in 3\Z$, there exists an oriented graph $G$ on $n$ vertices without a perfect transitive triangle tiling in which every vertex has both indegree and outdegree at least $\lceil 7n/18\rceil - 1.$
\end{abstract}

\section{Introduction}

Let $G$ be an oriented graph, that is
a directed graph without loops such that between every two vertices 
there is at most one edge.
We write $xy$ for an edge directed from $x$ to~$y$.
The \emph{outdegree} $d^+_G(x)$ of a vertex $x$ is the number of vertices $y$ such that $xy \in E(G)$. 
Similarly, the \emph{indegree} $d^-_G(x)$ of a vertex $x$ is the number of vertices $y$ such that $yx \in E(G)$. 
Define the \emph{minimum outdegree} $\delta^+(G)$ of $G$ to be the minimal $d^+_G(x)$ over all vertices $x$ of~$G$, and define the \emph{minimum indegree} $\delta^-(G)$ of $G$ similarly. 
Define the \emph{minimum semidegree} $\delta^0(G)$ of $G$ to be $\min \{ \delta^+(G) , \delta^-(G) \}$.

The oriented graph on $\{v_1, \dotsc, v_n\}$ with edge set 
$\{ v_nv_1 \} \cup \{ v_iv_{i+1} : i \in \{1, \dotsc, n - 1\} \}$ is
the \emph{directed cycle} of length $n$.
An oriented graph in which there is exactly one edge between every pair
of vertices is called a \emph{tournament}.
A tournament that does not contain a directed cycle is  \emph{transitive}.
Up to isomorphism, there are two tournaments on $3$ vertices: 
The directed cycle of length $3$, which we refer to as
the \emph{cyclic triangle}, and the transitive tournament on
$3$ vertices, which we refer to as the \emph{transitive
triangle} or as $TT_3$.

A \emph{tiling} of $G$ is a collection of vertex-disjoint subgraphs
called \emph{tiles}.
If every tile is isomorphic to some oriented graph $H$, then the
tiling is an $H$-\emph{tiling}.
If every vertex in $G$ is contained in a tile,
then the tiling is \emph{perfect}. 
The same definitions are applied to graphs and directed graphs.

In \cite{hajnal1970pcp},
Hajnal and Szemer{\'e}di proved that 
for any $k, r \in \N$ and for any graph $G$ on 
$kr$ vertices if the minimum degree
of $G$ is at least $(r-1)k$, then
$G$ has a perfect $K_r$-tiling.
The case when $r = 3$ was proved earlier by
Corr{\'a}di and Hajnal~\cite{corradi1963maximal}.

The  problem of finding
cyclic triangle tilings in an oriented graph was considered by Keevash and Sudakov~\cite{keevash2009},
who proved  a nearly optimal result:
      For some $\varepsilon > 0$ there exists $n_0$ such that
      if $G$ is an oriented graph on $n \ge n_0$ vertices and 
      $\delta^0(G) \ge (1/2 - \varepsilon)n$, then $G$ contains a
      cyclic triangle tiling that covers all but at most $3$ vertices.
      Furthermore, 
      if $n \equiv 3 \pmod{18}$, then there is a tournament $T$ such
      that $\delta^0(T) \ge (n - 1)/2 - 1$ 
      which does not have a perfect cyclic triangle tiling.
They repeated the following question which was asked by both
Cuckler \cite{cuckler} 
and Yuster \cite{yuster}.
\begin{qns}
  Does every tournament $T$ on $n \equiv 3 \pmod 6$ vertices
  with $\delta^0(T) = (n-1)/2$
  have a perfect cyclic triangle tiling?
\end{qns}
In this paper, we consider the problem of finding a perfect transitive triangle tiling,
proving an analogue to Corr{\'a}di and Hajnal's result for oriented graphs.
\begin{thm} \label{thm:TT3tiling}
  There exists $n_0$ such that for every $n \in 3\Z$ when
  $n \ge n_0$ the following holds.
  If $G$ is an oriented graph on $n$ vertices and $\delta^0(G) \ge 7n/18$,
  then $G$ contains a perfect $TT_3$-tiling.
\end{thm}
Treglown \cite{treglown_note} conjectured that
Theorem~\ref{thm:TT3tiling} is true for every $n \in 3\Z$.

The related problems for directed graphs have been considered 
(see \cite{wang00}, \cite{ckm2013}, \cite{cdkm2013} and \cite{treglown_hsz}).

\smallskip

The following family of
examples, 
based on the example of Treglown in \cite{treglown_note},
shows that Theorem~\ref{thm:TT3tiling} is tight.
For every positive $n \in 3 \Z$,
let $G$ be an oriented graph on $n$ vertices and 
%we construct 
%let $G$ be the following 
%oriented graph on $n$ vertices.
%Let $W_1,W_2,W_3,U_1,U_2$ be a partition of $V(G)$ such that 
$W_1,W_2,W_3,U_1,U_2$ be a partition of $V(G)$ such that
%Note that $\delta^0(G) =  \lfloor 2n/9 \rfloor + \lfloor (n-3)/6 \rfloor + 1$,
%given by the following table
\begin{align*}
|W_i| & = \left\lfloor \frac{2n/3 +i}3 \right\rfloor \text{for $i \le 3$},&
|U_1| &= \left\lfloor \frac{n-3}6 \right\rfloor, &
|U_2| & = \left\lceil \frac{n-3}6 \right\rceil.
\end{align*}
Let the edges of $G$ be all possible directed edges from $W_1$ to $W_2$, from $W_2$ to $W_3$, from $W_3$ to $W_1$, from $U_1$ to $U_2$, from $W_1 \cup W_2$ to $U_1$, from $U_1$ to $W_3$, from $U_2$ to $W_1 \cup W_2$ and from $W_3$ to $U_2$,
see Figure~\ref{fig:extremal}.
Note that, for every $v \in V(G)$,
because $|W_1| \ge |U_1|$ and $|W_2| \ge |U_2|$,
\begin{equation*}
  d^+(v), d^-(v) \ge \min\{|W_1| + |U_2|, |W_2| + |U_1|\},
\end{equation*}
so, for $w \in W_3$, 
\begin{equation*}
  \min\{d^+(w), d^-(w)\} = \min\{|W_1| + |U_2|, |W_2| + |U_1|\} = \delta^0(G) = 
  \ceiling{7n/8} - 1,
\end{equation*}
see Table~\ref{table:examples}.
\begin{table}
  \begin{equation*}
  \begin{array}{l|l|l|l|l|l|l|l}
    n    & |W_1| & |W_2| & |W_3| & |U_1| & |U_2| & \delta^0(G) & \ceiling{7n/18}\\
    \hline
    18m      & 4m     & 4m     & 4m + 1 & 3m - 1 & 3m    & 7m - 1 & 7m \\
    18m + 3  & 4m + 1 & 4m + 1 & 4m + 1 & 3m     & 3m    & 7m + 1 & 7m + 2\\
    18m + 6  & 4m + 1 & 4m + 2 & 4m + 2 & 3m     & 3m + 1 & 7m + 2 & 7m + 3\\
    18m + 9  & 4m + 2 & 4m + 2 & 4m + 3 & 3m + 1 & 3m + 1 & 7m + 3 & 7m + 4\\
    18m + 12 & 4m + 3 & 4m + 3 & 4m + 3 & 3m + 1 & 3m + 2 & 7m + 4 & 7m + 5\\ 
    18m + 15 & 4m + 3 & 4m + 4 & 4m + 4 & 3m + 2 & 3m + 2 & 7m + 5 & 7m + 6.
  \end{array}
  \end{equation*}
  \caption{The order of the sets $W_1, W_2, W_3, U_1$ and 
    $U_2$, $\delta^0(G)$ and 
    $\ceiling{7n/8}$ in the extremal graph for all values of $n \pmod{18}$.}
  \label{table:examples}
\end{table}
%Note that $\delta^0(G) =  \lfloor 2n/9 \rfloor + \lfloor (n-3)/6 \rfloor + 1$,
%which is achieved by considering a vertex in $W_3$
%(the outdegree when $n \equiv 15 \pmod{18}$ and indegree in all other cases).
Note that $G[W_1 \cup W_2 \cup W_3]$ does not contain a transitive triangle, so
every transitive triangle in $G$ contains a vertex in $U_1 \cup U_2$. Therefore,
the fact that $|U_1 \cup U_2| < n/3$ implies that $G$ does not contain a perfect $TT_3$-tiling.
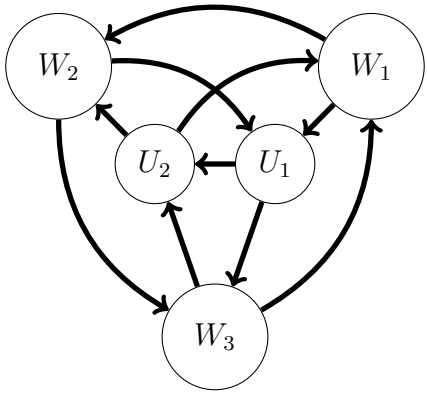
\begin{figure}[tp]
\centering
\begin{tikzpicture}[scale=0.2]
			\node[draw, circle, minimum size = 40]  (w1) at (30:12)  {$W_1$};
			\node[draw, circle, minimum size = 40]  (w2) at (150:12)  {$W_2$};
			\node[draw, circle, minimum size = 40]  (w3) at (-90:12) {$W_3$};
			\node[draw, circle, minimum size = 30]  (u1) at (4,-0.5)  {$U_1$};
			\node[draw, circle, minimum size = 30]  (u2) at (-4,-0.5)  {$U_2$};
			
			\begin{scope}[->, line width=2pt]
			\draw (w1) to [bend right] (w2);
			\draw (w2) to [bend right] (w3);
			\draw (w3) to [bend right] (w1);
			\draw (u1) -- (u2);
			\draw (u1) -- (w3);
			\draw (w1) -- (u1);
			\draw (w2) to [bend left] (u1);
			\draw (w3) -- (u2);
			\draw (u2) to [bend left] (w1);
			\draw (u2) -- (w2);			
			\end{scope}
\end{tikzpicture}
\caption{The extremal graph.}
\label{fig:extremal}
\end{figure}
 
\subsection{Outline of the paper}

We prove Theorem~\ref{thm:TT3tiling} using a stability approach and the
absorption technique.
We say that an oriented graph $G$ on $n$ vertices is 
\emph{$\alpha$-extremal} if there exists $W \subseteq V(G)$
such that $|W| \ge (2/3 - \alpha)n$ and $G[W]$ does not
contain a transitive triangle.

In Section~\ref{sec:non-extremal}, we handle the case when 
 $G$ is not  {$\alpha$-extremal}, i.e.\ we prove the following lemma.

%\begin{restatable}{lma}{nonextremal} 
 \begin{lma}
  \label{lma:nonextremal}
  For every $\alpha >0$ there exists $\varepsilon = \varepsilon(\alpha) > 0$
  and $n_0 = n_0(\alpha)$ such that when $G$ is
  an oriented graph on $n \in 3 \Z$ vertices and $n \ge n_0$
  the following holds.
  If $\delta^0(G) \ge (7/18 - \varepsilon)n$, then 
  $G$ has a perfect $TT_3$-tiling or $G$ is $\alpha$-extremal.
\end{lma}
%\end{restatable}

In Section~\ref{sec:extremal}, we prove Theorem~\ref{thm:TT3tiling} for
  oriented graphs $G$ which are {$\alpha$-extremal}.

%\begin{restatable}{lma}{extremal}
  \begin{lma}
  \label{lma:extremal}
  There exists $\alpha > 0$ and $n_0$ such that 
  when $G$ is an oriented graph on $n \in 3 \Z$ vertices
  and $n \ge n_0$ the following holds.
  If $\delta^0(G) \ge {7n}/{18}$ and
  $G$ is $\alpha$-extremal, 
  then there exists a perfect $TT_3$-tiling of $G$.
\end{lma}
%\end{restatable}

Lemma~\ref{lma:nonextremal} and Lemma~\ref{lma:extremal}
together clearly prove Theorem~\ref{thm:TT3tiling}.

While proving Lemma~\ref{lma:nonextremal} we prove the following
 result which may be of some interest because it applies for
all $n$.  Furthermore, it might be possible
to extend the proof of this theorem 
to prove the main theorem for all $n$.
%\begin{restatable}{thm}{neartiling}
\begin{thm}
  \label{thm:near_tiling}
  If $G$ is an oriented graph on $n$ vertices
  and $\delta^0(G) \ge {7n}/{18}$,
  then there exists a $TT_3$-tiling of $G$ that
  covers all but at most $11$ vertices.
\end{thm}
%\end{restatable}

\subsection{Notation}
Given a graph or digraph $G$, we write $V(G)$ for its vertex set, $E(G)$ for its edge set, and $e(G)=|E(G)|$ for the number of its edges.
% , $e(G)=|E(G)|$ for the number of its edges and $|G|=|V(G)|$ for the number of its vertices. 
Given a collection $\mathcal{T}$ of subgraphs, 
we write $V(\mathcal{T})$ for $\bigcup_{T \in \mathcal{T}} V(T)$.
When $\mathcal{W}$ is a collection of vertex subsets we will also use the
notation $V(\mathcal{W})$ to denote $\bigcup_{W \in \mathcal{W}} W$.

Suppose that $G$ is an oriented graph.
If $x$ is a vertex of $G$, then $N^+_G(x)$ denotes the \emph{out-neighborhood} of $x$, i.e.\ the
set of all those vertices $y$ for which $xy\in E(G)$.
Similarly, $N^-_G(x)$ denotes the \emph{in-neighborhood} of $x$, i.e.\ the set of all those vertices $y$ for which $yx\in E(G)$.
Note that  $d^+_G(x) =|N^+_G(x)|$ and $d^-_G(x) = |N^-_G(x)|$. 
We write $N_G(x) = N^+_G(x) \cup N^-_G(x)$ and $d_G(x) = d^+_G(x) + d^-_G(x)$.
We write $\delta(G)$ and $\Delta(G)$ for the minimum degree and
maximum degree of the 
underlying undirected graph of $G$, respectively.
Given a vertex $v$ of $G$ and a set $A\subseteq V(G)$, we define 
$N^+_G(v, A) = N^+_G(v) \cap A$ and $d^+_G(v,A) = |N^+_G(v, A)|$ and define $N^-(v, A)$, $d^-_G(v,A)$, $N_G(v,A)$ and $d_G(v,A)$ similarly.
Given $A,B \subseteq V(G)$, 
let $\overrightarrow{E}_G(A,B)$ be the set of
edges in $G$ directed from $A$ to $B$.
%  and let $\overrightarrow{e_G}(A,B) = |\overrightarrow{E}_G(A, B)|$.
Similarly, $E_G(A,B)$ denotes the set of
edges with one endpoint  in $A$ and the other in $B$  
and let $e_G(A, B) = |E_G(A, B)|$.
For a vertex $v$, we write $E_G(v)$ for $E_G(v, V(G))$.
For a vertex set $A \subseteq V(G)$, we write $G[A]$ for the subgraph of $G$
induced by $A$ and let $e_G(A) = e(G[A])$.
If $G$ is known from the context, then we may omit the subscript.
We let $G[A, B]$ be the bipartite graph in which $a \in A$ is adjacent
to $b \in B$ if and only if $ab \in E(G)$ or $ba \in E(G)$.
If $x,y,z \in V(G)$ 
we sometimes refer to $G[\{x, y, z\}]$ as $xyz$ or as $xe$,
where $e = yz$ or $e = zy$. 
If we refer to a directed path or a cyclic triangle as $xyz$,
then it must contain the edge set $\{xy, yz\}$ or
the edge set $\{xy, yz, zx\}$, respectively.

For   $m \in \N$ we write
$[m] = \{1, \dotsc, m\}$.
For any set $V$ we let $\binom{V}{m}$ be the collection of
subsets of $V$ that are of order $m$.
When it is clear that a variable $i$ must remain in $[m]$
(e.g. when $i$ is the index of $W_1, \dotsc, W_m$)
we let $i + 1 = 1$ when $i=m$ and $i - 1 = m$ when $i=1$.

%At various points in the proof it will be convenient to consider
%the following auxiliary bipartite graph.
%Let $G$ be an oriented graph and define 
%$B(G)$ to be  the $(E(G), V(G))$-bipartite graph in which
%$e \in E(G)$ is adjacent to $v \in V(G)$ if and only if 
%$ev$ is a transitive triangle.

\subsection{Preliminary lemmas and propositions}

Let $G$ be an oriented graph. 
Note that if $uw$ is an edge in $G[N^+(v)]$, then $vuw$ is a 
transitive triangle.
We get the following easy proposition.
\begin{prp} \label{prp:TT4}
Let $G$ be a tournament on $4$ vertices.
Then every vertex of $G$ is contained in a transitive triangle. 
\end{prp}

% \begin{prp} %\label{prp:TT3}
% For every $\eps>0$, there exist $\rho , n_0 >0$ such that if $G$ is an oriented graph on $n \ge n_0$ vertices with $e(G) \ge (2/3 + \eps)\binom{n}{2}$, then $G$ contains at least $\rho n^3$ transitive triangles.
% \end{prp}

% \begin{proof}
% By the supersaturation phenomenon discovered by Erd\H{o}s and Simonovits~\cite{MR726456}, $G$ contains at least $\rho n^4$ tournament on $4$ vertices (as $\text{ex}( n, K_4) \le 2\binom{n}{2}/3$).
% Since every tournament on $4$ vertices contains a transitive triangle by Proposition~\ref{prp:TT4}, the proposition follows.
% \end{proof}

\begin{prp} \label{prp:deg}
Let $G$ be an oriented graph on $n$ vertices.
Then
\begin{itemize}
	\item[{\rm (a)}] every (directed) edge $uv$ is contained in at least $3 \delta^0 (G) - n $ transitive triangles $uvw$ such that $w \in N^{-}(v)$;
	\item[{\rm (b)}] every (directed) edge $uv$ is contained in at least $3 \delta^0 (G) - n $ transitive triangles $uvw$ such that $w \in N^{+}(u)$;
	\item[{\rm (c)}] for every directed path $uvw$ on $3$ vertices, there are at least $2(3 \delta^0 (G) - n )$ vertices $x$ such that there exists a transitive triangle in $G[\{u,v,w,x\}]$ containing $x$ and $v$.
\end{itemize}
\end{prp}

\begin{proof}
Let $uv$ be an edge in $G$.
Note that every vertex $w$ in $N(u) \cap N^-(v)$ forms a transitive triangle with $uv$. 
Since $|N(u) \cap N^-(v)| \ge \delta (G) + \delta^-(G) -n \ge 3 \delta^0 (G) - n $, (a) follows.
By a similar argument, (b) also holds.

Let $uvw$ be a directed path on $3$ vertices.
By (a), there is a set $U \subseteq N^-(v)$ with $|U| \ge 3 \delta^0 (G) - n$ such that every $u' \in U$ forms a transitive triangle with $uv$.
By (b), there is a set $W \subseteq N^+(v)$ with $|W| \ge 3 \delta^0 (G) - n$ such that every $w' \in W$ forms a transitive triangle with $vw$.
Since $U \cap W = \emptyset$, (c) holds.
\end{proof}

\section{Non-extremal case}
\label{sec:non-extremal}

\subsection{Absorbing structure}

In this section, we prove Lemma~\ref{lma:absorption}.
Roughly speaking, the lemma states that there exists a small vertex set $U \subseteq V(G)$ such that $G[U \cup W]$ contains a perfect $TT_3$-tiling for every small $W \subseteq V(G) \setminus U$.
Thus, in order to find a perfect $TT_3$-tiling in $G$, it is suffices to find a $TT_3$-tiling covering almost all vertices in $G[V(G) \setminus U]$.
This technique was introduced by R\"odl, Ruci\'{n}ski and Szemer\'{e}di~\cite{MR2500161} to obtain results on matchings in hypergraphs.

For any $r \in \N$ and
any collection $\s{H}$ of oriented graphs on $[r]$,
define $\s{F}(\s{H}, G)$ to be the set of functions $f$
from $[r]$ to $V(G)$ such that $f$ is a directed graph homomorphism from
some $H \in \s{H}$ to $G$.
Let $\s{K}$ be the set of oriented graphs $K$ on $\{1,\dotsc,21\}$ such that both
$K$ and $K[\{1, \dotsc, 18\}]$ have a perfect $TT_3$-tiling.
For any ordered triple $X = (x_1, x_2, x_3)$ of vertices in $G$,
let $\s{A}'(X)$ be the set of functions $f \in \s{F}(\s{K}, G)$
such that $f(19) = x_1$, $f(20) = x_2$ and $f(21) = x_3$.
Let $\s{A}(X)$ be the set of functions 
in $\s{A}'(X)$ restricted to $[18]$.
Clearly $|\s{A}(X)| = |\s{A}'(X)|$.
Note that we do not require the functions in $\s{F}(\s{H}, G)$ to
be injective, but at a later stage of the proof
non-injective functions will essentially be discarded.
We consider non-injective functions only 
to make the following arguments simpler.

\begin{lma}
  \label{lma:absorbX}
  For $\eps_0 = 1/250$
  and $0 \le \varepsilon \le \eps_0$, 
  there exists $\tau = \tau(\varepsilon) > 0$ and 
  $n_0 = n_0(\varepsilon)$ such that the following holds.
  If $G$ is an oriented graph on $n \ge n_0$ vertices
  and $\delta^0(G) \ge (7/18 - \varepsilon)n$, then
  $|\s{A}(X)| \ge \tau n^{18}$ for every ordered triple $X = (x_1,x_2,x_3)$ of vertices in $G$.
\end{lma}
\begin{proof}
  Let $0 < \beta < (1/249 - \varepsilon)/10$ and $\tau = \beta^{18}$.
  Let $\s{T}$ be the set of functions from $\{1,2,3\}$ to $V(G)$
  that are digraph homomorphisms from a transitive triangle on $\{1,2,3\}$
  to~$G$. In other words, $\s{T}$ contains all functions from $\{1,2,3\}$ to 
  $V(G)$ whose image induces a transitive triangle.
  If we let $f(1)$ be any vertex $a \in V(G)$ and let $f(2)$ be any
  $b \in N_G(a)$, by Proposition~\ref{prp:deg},
  there are $3\delta^0(G) - n \ge (1/6 - 3\varepsilon)n$ 
  vertices we can assign to $f(3)$ so that $f \in \s{T}$. 
  This gives us that
  \begin{equation}
    \label{eq:number_of_TT3s}
    |\s{T}| \ge n \cdot (7/9 - 2 \varepsilon)n 
    \cdot (1/6 - 3\varepsilon)n > n^3/9 > (\beta n)^3.
  \end{equation}

%  Let $\s{L}_0$ be the set consisting  
%  of the unique oriented graph with vertex set $\{ 1 \}$, and
  For any $p \ge 1$,
  let $\s{L}_p$ be the set of oriented graphs $L$ on $[3p+1]$
  such that both $L[\{2, \dotsc ,3p + 1\}]$ and 
  $L[\{1, \dotsc, 3p\}]$
  have perfect $TT_3$-tilings
  (see Figure~\ref{fig:L_1_and_L_2} for some examples).
  For any $p \ge 1$ and $x,y \in V(G)$ (we allow $x = y$), 
  we let $\s{C}_p(x,y)$ be the set of $f \in \s{F}(\s{L}_p, G)$
  such that $f(1) = x$ and $f(3p+1) = y$,
%  We say that $x$ and $y$ are $0$-linked if $|\s{C}_0(x,y)| = 1$, so
%  $x$ and $y$ are \textit{$0$-linked} if and only if $x = y$.
  and we say that $x$ and $y$ are 
  \textit{$p$-linked} if 
  $|\s{C}_p(x,y)| \ge (\beta n)^{3p - 1}$.

  \begin{figure}[b]
    \begin{minipage}{0.45\textwidth}
      \centering
      \begin{tikzpicture}
        [
          vertex/.style={circle,fill=black,minimum size=6pt,inner sep=0pt}
        ]

        \node [vertex,label=above:{$1$}] (v1) {};
        \node [vertex,above right=of v1,label=above:{$2$}] (v2) {};
        \node [vertex,below right=of v1,label=below:{$3$}] (v3) {};
        \node [vertex,above right=of v3,label=above:{$4$}] (v4) {};
        \begin{scope}[->,thick,>=stealth]
          \draw (v1) to (v2);
          \draw (v1) to (v3);
          \draw (v2) to (v3);
          \draw (v4) to (v3);
          \draw (v2) to (v4);
        \end{scope}
      \end{tikzpicture}
      \hspace{0.5in}
      \begin{tikzpicture}
        [
          vertex/.style={circle,fill=black,minimum size=6pt,inner sep=0pt}
        ]
        \node [vertex,label=above:{$1$}] (v1) {};
        \node [vertex,above right=20pt of v1,label=above:{$3$}] (v3) {};
        \node [vertex,below right=20pt of v3,label=below:{$4$}] (v4) {};
        \node [vertex,above right=20pt of v3,label=above:{$5$}] (v5) {};
        \node [vertex,below=20pt of v4,label=below:{$2$}] (v2) {};
        \node [vertex,below right=20pt of v5,label=above:{$6$}] (v6) {};
        \node [vertex,below right=20pt of v6,label=above:{$7$}] (v7) {};
        \begin{scope}[->,thick,>=stealth]
          \draw (v1) to (v2);
          \draw (v1) to (v3);
          \draw (v3) to (v2);
          \draw (v3) to (v5);
          \draw (v4) to (v3);
          \draw (v6) to (v4);
          \draw (v2) to (v6);
          \draw (v2) to (v7);
          \draw (v7) to (v6);
          \draw (v6) to (v5);
          \draw (v4) to (v5);
        \end{scope}
      \end{tikzpicture}
      \caption{A graph in $\s{L}_1$ and a graph in $\s{L}_2$.}
      \label{fig:L_1_and_L_2}
    \end{minipage}
    \begin{minipage}{0.45\textwidth}
      \centering
      \begin{tikzpicture}
        \node [draw, rectangle, minimum size=70pt] (A1) {$\npart{+}{+}(x,y)$};
        \node [draw, rectangle, minimum size=70pt,right=-.5pt of A1] (B2)
        {$\npart{-}{+}(x,y)$};
        \node [draw, rectangle, minimum size=70pt,below=-.5pt of A1] (B1)
        {$\npart{+}{-}(x,y)$};
        \node [draw, rectangle, minimum size=70pt,right=-.5pt of B1] (A2)
        {$\npart{-}{-}(x,y)$};
        \node [draw, draw opacity=0, circle, minimum size=55pt, at=(A1)] (A1C) {};
        \node [draw, draw opacity=0, circle, minimum size=55pt, at=(A2)] (A2C) {};
        \node [draw, draw opacity=0, circle, minimum size=55pt, at=(B1)] (B1C) {};
        \node [draw, draw opacity=0, circle, minimum size=55pt, at=(B2)] (B2C) {};
%	\node at ($(A1) + (0,35pt)$) {$N^+_{G}(x)$};
%	\node at ($(B2) + (0,35pt)$) {$N^-_{G}(x)$};
%	\node at ($(A1) - (45pt,0)$) {$N^+_{G}(y)$};
%	\node at ($(B1) - (45pt,0)$) {$N^-_{G}(y)$};
%        \node [draw, fill=black, circle, inner sep=0pt, minimum size=6pt, label=above:{$x$}]
%        at ($ (A1) !.5! (B2) + (0,40pt)$) (x) {};
%        \node [draw, fill=black, circle, inner sep=0pt, minimum size=6pt, label=left:{$y$}]
%        at ($ (A1) !.5! (B1) + (-40pt,0)$) (y) {};
        \begin{scope}[->, line width=2pt,>=stealth]
%          \draw [thick] (x) to (A1.north);
%          \draw [thick] (B2.north) to (x);
%          \draw [thick] (y) to (A1.west);
%          \draw [thick] (B1.west) to (y);
          \draw (A1C) to (B2C);
          \draw (A1C) to (B1C);
          \draw (A1C) to (A2C);
          \draw (B1C) to (A2C);
          \draw (B2C) to (A2C);
          \draw [<->] (B1C) to (B2C);
        \end{scope}
      \end{tikzpicture}
      \caption{The possible orientations of edges in $G'[N(x,y)]$.}
      \label{fig:partition}
    \end{minipage}
  \end{figure} 

  For any $q > p \ge 1$,
  $f \in C_p(x,y)$ and $g_i \in \mathcal{T}$ for $i \in [q - p]$,
  the function
  \begin{equation*}
    h(j) = 
    \begin{cases} 
      x &\text{if $j = 1$} \\
      f(j) &\text{if $2 \le j \le 3p$} \\
      g_i(k) &\text{if $j = 3p+3(i-1) + k$ for some $i \in [q-p]$ and $k \in
      [3]$} \\
      y &\text{if $j=3q+1$}
    \end{cases}
  \end{equation*}
  is in $\s{C}_q(x,y)$. 
  Therefore, 
  \begin{equation*}
    |\s{C}_q(x,y)| \ge |\s{C}_p(x,y)||\s{T}|^{q-p} \text{ for any $q > p$}.
  \end{equation*}
%  So
%  \footnote{Here I think a 1/2 factor is needed, check around Claim 1 as
%  well.}
  With \eqref{eq:number_of_TT3s}, this implies that,
  if $x$ and $y$ are $p$-linked,
  then $x$ and $y$ are also $q$-linked for any $q > p$.
  Recall that we do not require the functions in $\s{C}_q(x,y)$ to be
  injective.

  Let $X = (x_1, x_2, x_3)$ be an ordered triple of vertices in $G$,
  %We have that 
  %\footnote{the formula is not exactly correct, some sets could be
  %intersecting. So writing a 1/2 factor could make it more precise.}$
%  Indeed, 
  $f \in \s{T}$ and $g_i \in \s{C}_2(f(i), x_i)$ for $i \in [3]$.
  Define $h : [21] \to V(G)$ by
  \begin{equation*}
    h(j) = 
    \begin{cases} 
      g_i(k) &\text{if $j = 6(i-1) + k$ for some $i \in [3]$ and $k \in
      [6]$}\\
      x_i &\text{if $j=18 + i$ for some $i \in [3]$}.
    \end{cases}
  \end{equation*}
  By the definition of $\s{C}_2$ and the fact that
  the image $h(\{1, 7, 13\}) = f([3])$ induces a transitive triangle in $G$, we have that
  $h \in \s{A}'(X)$.
%  and that the inequality holds.
  Therefore, 
  \begin{equation*}
    |\s{A}(X)| = 
    |\s{A}'(X)| \ge
    \sum_{f \in \s{T}}\prod_{i \in [3]} |\s{C}_2(f(i), x_i)|,
  \end{equation*}
  so, with \eqref{eq:number_of_TT3s},
  we can complete the proof of the lemma by showing
  that every pair of vertices in $V(G)$ is $2$-linked. 

\smallskip
% \begin{clm}
%   If $x$ and $y$ ar
% \end{clm}

  The remainder of the proof relies on analyzing the intersection of the 
  neighborhood of two vertices in detail. 
  To facilitate this, we make the following definition
  and simple computations.
  For any $U \subseteq V(G)$, let $N_G(U) = \bigcap_{u \in U} N_G(u)$. 
  If $U = \{x,y\}$ is a $2$-set, we often write
  $N_G(x,y)$ instead of $N_G(U)$.
  We have the following inequality
  \begin{equation}
    \label{eq:intersection}
    N_G(U) \ge |U|\delta(G) - (|U| - 1)n \ge 
    \left(\frac{9 - 2|U|}{9} - 2|U|\varepsilon\right)n.
  \end{equation}

  For any pair $x,y \in V(G)$, let
  \begin{equation*}
    \begin{aligned}
      \npart{+}{+}(x,y) & = N^+_G(x) \cap N^+_G(y), \\
      \npart{+}{-}(x,y) & = N^+_G(x) \cap N^-_G(y), 
    \end{aligned}
    \qquad \qquad
    \begin{aligned}
      \npart{-}{+}(x,y) & = N^-_G(x) \cap N^+_G(y), \\
      \npart{-}{-}(x,y) & = N^-_G(x) \cap N^-_G(y) 
    \end{aligned}
  \end{equation*} 
  and $\s{N}(x,y) = \{
    \npart{+}{+}(x,y), 
    \npart{+}{-}(x,y), 
    \npart{-}{+}(x,y), 
    \npart{-}{-}(x,y)
  \}$.

  It will be important for us to know when two vertices are 
  $1$-linked, so we let $F(x,y)$ be the set of
  edges such that both $xe$ and $ye$ are transitive triangles,
  which means that every edge in $F(x,y)$ corresponds to 
  two distinct homomorphisms in $\s{C}_1(x,y)$, and 
  \begin{equation}
    \label{eq:1-linked}
    \text{if $|F(x,y)| \ge (\beta n)^2/2$, then
    $x$ and $y$ are $1$-linked. }
  \end{equation}

  Let $uv \in E(G[N(x,y)])$.
  If $u \in N^{-,-}(x,y)$ or $v \in N^{+,+}(x,y)$ or both $u$ and $v$ are in the same set $A \in \s{N}(x,y)$, then $uv \in F(x,y)$.
  Otherwise, $uv \notin F(x,y)$.
  Indeed, since $u \notin N^{-,-}(x,y)$, $u$ is an outneighbor of one of $x$ or $y$, say $x$.
  If we assume $xuv$ is a transitive triangle, then $v \in N^+(x)$.
  This implies that $v \in N^-(y)$, because $v \notin N^{+,+}(x,y)$,
  which further implies that $u \in N^{+}(y)$, because 
  $u$ and $v$ cannot both be in $N^{+,-}(x,y)$. 
  Therefore, $yuv$ is not a transitive triangle.
%  If $uv \in F(x,y) \setminus \bigcup_{A \in \s{N}(x,y)}E(G[A])$, then for one of $x$ or $y$, say $x$,
%  $u$ and $v$ cannot both be inneighbors of $x$ and 
%  cannot both be outneighbors of $x$.
%  Therefore, since $xuv$ is a transitive triangle,
%  $u \in N^-(x)$ and $v \in N^+(x)$.
%  So, because $yuv$ is a transitive triangle,
%  either $u \in \npart{-}{-}(x,y)$ or $v \in \npart{+}{+}(x,y)$.
  Hence,  
  \begin{equation}
    \label{eq:def_of_F}
    F(x,y) = \overrightarrow{E}_G(\npart{-}{-}(x,y), N(x,y)) \cup 
    \overrightarrow{E}_G(N(x,y), \npart{+}{+}(x,y)) \cup
    \bigcup_{A \in \s{N}(x,y)}E(G[A]).
  \end{equation}
  %\footnote{So this 1/2 should be compared with the other footnotes.} 

  \setcounter{clm}{0} 
  \begin{clm}
    \label{clm:large_quadrant}
    For any pair $u,w \in V(G)$, if there exists $A \in \s{N}(u,w)$ such that 
    \begin{equation*}
      |A| \ge (2/9 + \beta + 2\varepsilon)n,
    \end{equation*}
    then $u$ and $u$ are $1$-linked.
    In particular, for any $v \in V(G)$, the pair $(v,v)$ is $1$-linked.
  \end{clm}
  \begin{proof}
    Since, by \eqref{eq:def_of_F},
    \begin{equation*}
      |F(u,w)| \ge |E(G[A])| \ge |A|(\delta(G) + |A| -n)/2 > (\beta n)^2/2,
    \end{equation*}
    the claim follows from \eqref{eq:1-linked}.
  \end{proof}

  \begin{clm}
    \label{clm:large_intersection_1_close}
    If $|N(u,w)| \ge (11/18 + \beta + \varepsilon)n$, then 
    $u$ and $w$ are $1$-linked.
  \end{clm}
  \begin{proof}
%    Since $u$ and $w$ are both not in $N^0$, they both have all but less than $\beta n$
%    of their neighbors in $\overline{A} = V(G) \setminus A$, and we have that,
%    \begin{align*}
%      |N_G(u,w)| & > 2(\delta(G) - \beta n) - |\overline{A}| \\
%      & \ge (14/9 - 4\varepsilon - 2\beta)n - (17/18 - 5\varepsilon - 3\beta)n
%      = (11/18 + \beta + \varepsilon)n.
%    \end{align*}
    Let $B = \npart{+}{+}(u,w) \cup \npart{-}{-}(u,w)$ and let $v \in B$. 
    By~\eqref{eq:def_of_F}, if $v \in \npart{+}{+}(u,w)$ and $v' \in N_G(u,w) \cap N^-_G(v)$, then $v'v \in F(u,w)$, and if $v \in \npart{-}{-}(u,w)$ and $v' \in N_G(u,w) \cap N^+_G(v)$, then $vv' \in F(u,w)$.
    Therefore,
    \begin{equation*}
      |E_G(v) \cap F(u,w)| \ge \delta^0(G) + |N_G(u,w)| - n \ge \beta n.
    \end{equation*}
    Hence, if $|B| \ge \beta n$,
    then $|F(u,w)| \ge (|B| \cdot \beta n)/2 = (\beta n)^2/2$ and,
    by \eqref{eq:1-linked},
    $u$ and $w$ are $1$-linked.
    If $|B| < \beta n$,
    then there exists $C \in \{\npart{+}{-}(u,w), \npart{-}{+}(u,w)\}$ 
    such that
    %\footnote{I changed (x,y) to (u,v)}
    \begin{equation*}
      |C| \ge (|N_G(u,w)| - |B|)/2 > n/4,
    \end{equation*}
    which, with Claim~\ref{clm:large_quadrant}, 
    implies that $u$ and $w$ are $1$-linked.
  \end{proof}

  Assume that there exists a pair 
  $x,y \in V(G)$ that is not $1$-linked. % and define $F = F(x,y)$.
%  We then also have that $x$ and $y$ are not $1$-linked,
  We have that $|F(x,y)| < (\beta n)^2/2$ by \eqref{eq:1-linked}.
%  In order to make the remainder of the argument simplier, we 
%  would like to ignore the edges in $F$
%  and the vertices that are incident to a large 
%  number of edges in $F$.
%  To make this formal, we let 
  Let $G' = G - F(x,y)$ 
  (see Figure~\ref{fig:partition}),
  and 
  let 
  \begin{equation*}
    N^0 = \{v \in N(x,y) : |E_G(v) \cap F(x,y)| \ge \beta n\}
  \end{equation*}
  be the set of vertices in $N(x,y)$ incident to a significant
  number of edges in $F(x,y)$.
  Note that, since $(\beta n)^2/2 > |F(x,y)| \ge (|N^0| \cdot \beta n)/2$,
  \begin{equation}
    \label{eq:size_of_N0}
    |N^0| < \beta n.
  \end{equation}
%  $N = N(x,y) \setminus N^0$,
%  $\npart{i}{j} = \npart{i}{j}(x,y) \setminus N^0$ for every $1 \le i, j \le 2$,
%  and $\s{N} = \{ \npart{i}{j} : 1 \le i,j \le 2 \}$
%  Note that
%  \begin{equation}
%    \label{eq:deg_G_prime}
%    \delta(G') \ge (7/9 - 2\beta)|G'| \text{ and } \delta^0(G') \ge (7/18 - 2\beta)|G'|
%  \end{equation}

  Our goal now is to show that $x$ and $y$ must be $2$-linked.
  To achieve this, we will use the following two claims.  Let $\Gamma$ be the set of triples $(w_1, w_2, w_3) \in V(G)^3$
  such that, for some ordering $\{i,j,k\} = [3]$, $xw_iw_k$ and $yw_jw_k$ are transitive triangles and $w_i$ and $w_j$ are $1$-linked.
  \begin{clm}\label{clm:triples}
    If $|\Gamma| \ge (\beta n)^3$, then $x$ and $y$ are $2$-linked.
  \end{clm}
  \begin{proof}
    Let $(w_1, w_2, w_3) \in \Gamma$ with the required ordering $\{i,j,k\} = [3]$.
    There are at least $(\beta n)^2$ pairs $(u,v) \in V(G)^2$ 
    such that $w_iuv$ and $w_juv$ are transitive triangles, and for every such pair, 
    the $7$-tuple $(x, w_1, w_2, w_3, u, v, y)$ corresponds to a function in
    $\mathcal{C}_2(x,y)$.
    Therefore, $|\mathcal{C}_2(x,y)| \ge |\Gamma| \cdot (\beta n)^2$, 
    and the conclusion follows (see Figure~\ref{fig:L_1_and_L_2}).
  \end{proof}
  \begin{clm}
    \label{clm:in_out_close}
    If $u,w \in N(x,y)$ and 
    \begin{equation*}
      \min\{d^+(u, N(x,y)), d^-(u, N(x,y))\}, 
      \min\{d^+(w, N(x,y)), d^-(w, N(x,y))\}
      < 4 \beta n,
  \end{equation*}
    then $u$ and $w$ are $1$-linked. 
%    Every pair $u,w \in I$ is $1$-linked.
  \end{clm}
  \begin{proof}
    Let $\overline{N} = V(G) - N(x,y)$.
		By \eqref{eq:intersection},
    \begin{equation*}
      |\overline{N}| \le n - (5/9 - 4\varepsilon)n = (4/9 + 4 \varepsilon)n.
    \end{equation*}
		If $u$ and $w$ are as in the statement of the claim, then there exist $\sigma_u, \sigma_w \in \{+ ,-\}$ such that $d_G^{\sigma_{u}}(u,\overline{N}), d_G^{\sigma_{w}}(w,\overline{N}) > \delta^0(G) - 4 \beta n$.
		Note that $N^{\sigma_u, \sigma_w}(u,w) \in \mathcal{N}(u,w)$ satisfies
		\begin{align*}
		 |N^{\sigma_u, \sigma_w}(u,w)| & \ge  d_G^{\sigma_{u}}(u,\overline{N}) + d_G^{\sigma_{w}}(w,\overline{N}) -  | \overline{N} | 
		 >  2 (\delta^0(G) - 4 \beta n) -  (4/9 + 4 \varepsilon)n \\
		& \ge (7/9 - 2\varepsilon - 8 \beta) n - (4/9 + 4\varepsilon)n > n/4.
		\end{align*}
		Applying Claim~\ref{clm:large_quadrant} then completes the proof.
  \end{proof}
  By \eqref{eq:def_of_F}, every vertex in $\npart{+}{+}(x,y) \setminus N^0$ has at most $\beta n$ inneighbors in $N(x,y)$, and
  every vertex in $\npart{-}{-}(x,y) \setminus N^0$ has at most $\beta n$ outneighbors in $N(x,y)$.
  This and Claim~\ref{clm:in_out_close} imply that 
  \begin{equation}
    \label{eq:++_--_one_linked}
    \text{every pair of vertices in
      $\left(\npart{+}{+}(x,y) \cup \npart{-}{-}(x,y)\right) \setminus N^0$ is $1$-linked}.
\end{equation}

  Suppose that $|\npart{+}{+}(x,y)|,|\npart{-}{-}(x,y)| \ge 2 \beta n$, so
%  In this case, there are $(\beta n)^2$ pairs 
%  $(a,b) \in \npart{+}{+}(x,y) \times \npart{-}{-}(x,y)$ such that both $a$ and $b$ are not in $N^0$,
%  then we will show that there are at least $(\beta n)^5$ functions in $C_2(x,y)$, which is a contradiction.
  there are $(\beta n)^2$ ways to select 
  $a \in \npart{+}{+}(x,y) \setminus N^0$ and 
  $b \in \npart{-}{-}(x,y) \setminus N^0$.
  By \eqref{eq:++_--_one_linked}, any such $a$ and $b$ are $1$-linked.
  Let $c \in N_{G}(\{a,b,x,y\})$. 
  If $c \notin \npart{+}{-}(x,y) \cup \npart{-}{+}(x,y)$, then,
  by \eqref{eq:def_of_F},
  either
  $a,c \in \npart{+}{+}(x,y)$ and the edge between $a$ and $c$ is in $F(x,y)$
  or 
  $b,c \in \npart{-}{-}(x,y)$ and the edge between $b$ and $c$ is in $F(x,y)$.
  Recall that since both $a$ and $b$ are not in $N^0$, 
  they both are incident to at most $\beta n$ edges in $F(x,y)$.
  With \eqref{eq:intersection}, this gives us that 
  \begin{equation*}
    |N_{G}(a,b) \cap (\npart{+}{-}(x,y) \cup \npart{-}{+}(x,y))| \ge 
    |N_{G}(\{a,b,x,y\})| - 2\beta n 
    > \beta n,
  \end{equation*}
  and we can pick 
  $c \in N_{G}(a,b) \cap (\npart{+}{-}(x,y) \cup \npart{-}{+}(x,y))$ in one
  of $\beta n$ ways.
  If $c \in  \npart{+}{-}(x,y)$, then $xac$ and $ybc$ are transitive triangles,
  and if $c \in \npart{-}{+}(x,y)$, then $xbc$ and $yac$ are transitive triangles.
  Therefore, in any case, $(a,b,c) \in \Gamma$, which implies
  $|\Gamma| \ge (\beta n)^3$, so $x$ and $y$ are $2$-linked by Claim~\ref{clm:triples}.

  Therefore, we can assume that $\min\{|\npart{+}{+}(x,y)|,|\npart{-}{-}(x,y)|\} < 2 \beta n$, 
  and, by considering the graph in which all of the edge orientations
  are reversed, we can further assume that 
  \begin{equation*}
    |\npart{-}{-}(x,y)| < 2 \beta n.
  \end{equation*}
%  Let $j = +1$, $A = \npart{+}{+}(x,y) \setminus N^0$ and
%  $B = \npart{-}{-}(x,y)$ if 
%  $|\npart{-}{-}(x,y)| < 2 \beta n$
%  and otherwise, let $j = -1$, $A = \npart{-}{-}(x,y) \setminus N^0$.
%  and $B = \npart{+}{+}(x,y)$
  Note that, by \eqref{eq:intersection}, 
  Claim~\ref{clm:large_quadrant} and the fact that $x$ and
  $y$ are not $1$-linked,
  \begin{equation}
    \label{eq:plus_minus_large}
    |\npart{+}{-}(x,y)|,|\npart{-}{+}(x,y)| \ge 
    |N(x,y)| - |\npart{-}{-}(x,y)| - 2(2/9 + \beta + 2\varepsilon)n > n/12.
  \end{equation}
  Let $A \in \{\npart{+}{-}(x,y), \npart{-}{+}(x,y) \}$, and
  let $u,w \in A \setminus N^0$. 
  Since both $u$ and $w$ have all but less than $2\beta n$ of their
  neighbors in $V \setminus A$, using \eqref{eq:plus_minus_large}, we have that 
  \begin{equation*}
    |N_G(u,w)| \ge (14/9 - 4\varepsilon)n - 4 \beta n - (n - |A|) > (23/36 - 4(\beta + \varepsilon))n.
  \end{equation*}
  Claim~\ref{clm:large_intersection_1_close} and the fact
  that $5(\beta + \varepsilon) < 1/36$ then imply that 
  \begin{equation}
    \label{eq:plus_minus_1_close}
    \text{if $u,w \notin N^0$ are both in $\npart{+}{-}(x,y)$ or 
    both in $\npart{-}{+}(x,y)$,
    then $u$ and $w$ are $1$-linked.}
  \end{equation}
  By \eqref{eq:plus_minus_large}, 
  we can pick $a \in \npart+-(x,y) \setminus N^0$ in at least $\beta n$ ways.
  By \eqref{eq:intersection}, we have that $|N_G(x,y,a)| \ge (1/3 - 6 \varepsilon)n$, so
  since $a \notin N^0$ and $|\npart--(x,y)| < 2\beta n$, 
  $a$ has at least $n/4$ neighbors in $\npart++(x,y) \cup \npart-+(x,y)$.
  Therefore, Claim~\ref{clm:large_quadrant} and our assumption that $x$ and $y$ are not $1$-linked
  imply that
  \begin{equation}
    \label{eq:neighbors_of_a}
    \text{$a$ has at least $2 \beta n$ neighbors in both $\npart++(x,y)$ and $\npart-+(x,y)$.}
  \end{equation}
%  Claim~\ref{clm:neighbors_in_two}, the fact
%  that $|\npart--(x,y)| < 2\beta n$ and the fact that $a \notin N^0$,
%  imply that $a$ has at least $2 \beta n$ neighbors in $\npart-+(x,y)$.
  Using \eqref{eq:neighbors_of_a}, we can then select $b \in (N(a) \cap \npart-+(x,y)) \setminus N^0$ 
  in at least $\beta n$ ways.

%  We will now show that there exists at least $(\beta n)^3$ 
%  different triples $(c, u, w)$ 
%  so that each of the at least $(\beta n)^5$ choices for $a,b,c,u$ and $w$ 
%  correspond to a different element of $\mathcal{C}_2(x,y)$, 
%  which is a contradiction to the fact that $x$ and $y$
%  are not $2$-linked.

  Assume $ab \in E$, which implies that $yab$ is a transitive triangle.
  We will show that either there are 
  at least $\beta n$ vertices $c$ such that $c$ and $a$ are $1$-linked
  and $xbc$ is a transitive triangle, or
  there are at least
  $\beta n$ vertices $c$ such that $c$ and $b$ are $1$-linked
  and $xac$ is a transitive triangle,
  so, in either case, $(a,b,c) \in \Gamma$.
  By the symmetry of $x$ and $y$, the same argument shows that
  if $ba \in E$, then there are $\beta n$ vertices $c$ such that
  $(a,b,c) \in \Gamma$.
  This will complete the proof as 
  $|\Gamma| \ge (\beta n)^3$ and Claim~\ref{clm:triples} 
  imply that $x$ and $y$ are $2$-linked.

  Suppose that $b$ has at least $4 \beta n$ outneighbors in $N(x,y)$.
  This, with the fact that $b \notin N^0$, 
  implies that $|N^+_{G'}(b, N(x,y))| \ge 3 \beta n$.
  By \eqref{eq:def_of_F}, 
  $b \in \npart{-}{+}(x,y)$ implies that 
  $N^+_{G'}(b, N(x,y))$ is contained in 
  $\npart--(x,y) \cup \npart+-(x,y)$, so since 
  $|\npart{-}{-}(x,y)| < 2 \beta n$, there exist 
  at least $\beta n$ outneighbors $c$ of $b$ in 
  $\npart{+}{-}(x,y) \setminus N^0$.
  This completes the case, as for every such $c$,
  $xbc$ is a transitive triangle, and,
  since $a,c \in \npart{+}{-}(x,y) \setminus N^0$,
  $c$ and $a$ are $1$-linked by \eqref{eq:plus_minus_1_close}.

  Otherwise, $b$ has less than $4 \beta n$ outneighbors in $N(x,y)$.
  Since, by \eqref{eq:def_of_F}, every vertex in $\npart++(x,y) \setminus N^0$ has at most 
  $\beta n$ inneighbors in $N(x,y)$, 
  Claim~\ref{clm:in_out_close} implies that 
  $b$ is $1$-linked with every vertex in $\npart++(x,y) \setminus N^0$, 
  and, by \eqref{eq:neighbors_of_a}, $a$ has $\beta n$
  neighbors in $c \in \npart++(x,y) \setminus N^0$.
  For every such $c$, $xac$ is a transitive triangle and $b$ and $c$ are $1$-linked,
  which completes the case and the proof.
%  at least $(\beta n)^2$ pairs $u,w$ such that
%  $auw$ and $cuw$ are both transitive triangles.
%  Therefore, because $xbc$ and $yab$ are transitive triangles, 
%  there are at least $(\beta n)^3$ of the desired triples $(c,u,w)$.
\end{proof}

%  \begin{clm}
%    \label{clm:large_quadrant_1_close}
%    If $A \in \s{N}(x,y)$ and $|A| \ge (1/18 + 5\varepsilon + 3\beta)n$,
%    then every pair $u,w \in A \setminus N^0$ is $1$-linked.
%  \end{clm}
%  \begin{proof}
%    Since $u$ and $w$ are both not in $N^0$, they both have all but less than $\beta n$
%    of their neighbors in $\overline{A} = V(G) \setminus A$, and we have that,
%    \begin{align*}
%      |N_G(u,w)| & > 2(\delta(G) - \beta n) - |\overline{A}| \\
%      & \ge (14/9 - 4\varepsilon - 2\beta)n - (17/18 - 5\varepsilon - 3\beta)n
%      = (11/18 + \beta + \varepsilon)n.
%    \end{align*}
%    Let $B = \npart{+}{+}(u,w) \cup \npart{-}{-}(u,w)$ and let $v \in B$. 
%    By~\eqref{eq:def_of_F}, if $v \in \npart{+}{+}(u,w)$ and $v' \in N_G(u,w) \cap N^-_G(v)$, then $v'v \in F(u,w)$, and if $v \in \npart{-}{-}(u,w)$ and $v' \in N_G(u,w) \cap N^+_G(v)$, then $vv' \in F(u,w)$.
%    Therefore,
%    \begin{equation*}
%      |E_G(v) \cap F(u,w)| \ge \delta^0(G) + |N_G(u,w)| - n \ge \beta n.
%    \end{equation*}
%    Hence, if $|B| \ge \beta n$,
%    then $|F(u,w)| \ge (|B| \cdot \beta n)/2 = (\beta n)^2/2$ and,
%    by \eqref{eq:1-linked},
%    $u$ and $w$ are $1$-linked.
%    If $|B| < \beta n$,
%    then there exists $C \in \{\npart{+}{-}(u,w), \npart{-}{+}(u,w)\}$ 
%    such that
%    %\footnote{I changed (x,y) to (u,v)}
%    \begin{equation*}
%      |C| \ge (|N_G(u,w)| - |B|)/2 > n/4,
%    \end{equation*}
%    which, with Claim~\ref{clm:large_quadrant}, 
%    implies that $u$ and $w$ are $1$-linked.
%  \end{proof}

\begin{lma}[Absorbing Lemma]
  \label{lma:absorption}
  For every $0 \le \eps \le 1/250$, there exists 
  $\sigma_0 = \sigma_0(\eps)$ such that
  for every $0 <\sigma < \sigma_0$, there exists $n_0 = n_0(\eps, \sigma)$ such that the following holds.
  If $G$ is an oriented graph on $n \ge n_0$ vertices with 
  $\delta^0(G) \ge (7/18 - \eps)n$,
  then $G$ contains a vertex set $U \subseteq V(G)$ 
  with $|U| \le 3 \sigma n$ and $ |U| \in 3 \mathbb{Z}$ such that, 
  for every $W \subseteq V(G) \setminus U$ with 
  $|W| \le 3 \sigma^2 n $ and 
  $|W| \in 3 \mathbb{Z}$, $G[U \cup W]$ contains a perfect $TT_3$-tiling.
\end{lma}
\begin{proof}
  Let $\tau = \tau(\varepsilon)$ be the constant given
  by Lemma~\ref{lma:absorbX} and let $\sigma_0 = \tau/(72^2 + 1)$
  and let $0 < \sigma < \sigma_0$.
  Let $G$ be sufficiently large oriented graph with
  $\delta^0(G) \ge (7/18 - \varepsilon)n$.
  Let $\s{F}$ be the set of functions from $[18]$ to $V(G)$.
  Call a map $f \in \s{F}$ \textit{absorbing} if there exists
  an ordered triple $X$ of vertices such that $f \in \s{A}(X)$.

  Choose $\s{U}' \subseteq \s{F}$ by selecting
  each $f \in \s{F}$ independently at random with probability 
  $p = 2\sigma n^{-17}$. 
  Call a pair $f, g \in \s{F}$ \textit{bad} if either $f$ or $g$
  is not injective or the images of $f$ and $g$ intersect
  and note that 
  there are less than $n \cdot \binom{36}{2} \cdot n^{34}$ bad pairs in $\s{F}$.
  Therefore, the expected number of bad pairs in $\mathcal{U'}$ 
  is less than $18 \cdot 35 \cdot 4 \sigma^2 n$. 
  Thus, using Markov's inequality, we derive that, with probability
  more than $1/2$, 
  $\mathcal{U'}$ contains at most $(72 \sigma)^2 n$ bad pairs.

  By Chernoff's bound, the union bound 
  and Lemma~\ref{lma:absorbX},
  with positive probability the 
  set $\mathcal{U}'$ also satisfies 
  $|\mathcal{U}'| \le 3 \sigma n$ 
  and $|\s{A}(X) \cap \s{U}'| \ge \tau \sigma n$
  for each ordered triple $X$ of vertices. 
  By deleting every bad pair from $\mathcal{U}'$ 
  and any $f \in \mathcal{U}'$ for which $f$ is not absorbing, 
  we get $\mathcal{U}\subseteq \mathcal{U}'$ consisting of 
  injective homomorphisms with pairwise disjoint images.
  Moreover, for each ordered triple $X$ of vertices, there are at 
  least $\tau \sigma n - (72 \sigma)^2 n > \sigma^2 n$ 
  functions in $\s{A}(X) \cap \mathcal{U}$.
  Let $U$ be the union of the images of every $f \in \mathcal{U}$.
  Since $\mathcal{U}$ consists only of absorbing functions, 
  $G[U]$ has a perfect $TT_3$-tiling, so 
   $|U| \in 3 \mathbb{Z}$.
  For any set $W \subseteq V \backslash U$ of size $|W| \le 3\sigma^2 n$ and 
  $|W| \in 3 \mathbb{Z}$, $W$ can be partitioned into at most $\sigma^2 n$ sets
  of size $3$.
  Each such set can be arbitrarily ordered to give a triple $X$
  which then can be successively paired up with a different absorbing 
  homomorphism $f \in \s{A}(X) \cap \s{U}$. 
  Therefore, $G[U \cup W]$ contains a perfect $TT_3$-tiling. 
\end{proof}

%%%%%%%%%%%%%%%%%%%%%%%%%%%%%%%%%%%%%%%%%%%%%%%%%%%%%%%%%%%%%%%%%%%%%%%%%%%%%%%%%%%%%%%%%%%%%%%%%%%%%%%%
\subsection{Almost $TT_3$-tilings}
\label{sec:almost}

%Given $\alpha>0$, an oriented graph $G$ on $n$ vertices is said to be
%\emph{$\alpha$-extremal} if there exists a vertex set $W \subseteq V(G)$ with
%$|W| \ge 2(1- \alpha)  n/3$ for which there exists a spanning subgraph $G'$ of
%$G[W]$ such that 
%$G'$ is a subgraph of a blow-up of a cyclic triangle of order $|W|$ and
%$e(G[W] - G') \le \alpha n^2$. 

Theorem~\ref{thm:near_tiling} and Lemma~\ref{lma:almost}, 
which we prove simultaneously,
show that if $\delta^0(G) \ge 7n/18$  then there is a
$TT_3$-tiling on all but at most $11$ vertices of $G$,
and if $\delta^0(G)$ is slightly less than  $7n/18$
then there is a $TT_3$-tiling on all but at most $14$ vertices or $G$ is $\alpha$-extremal for some small $\alpha>0$, respectively.

%\neartiling*

\begin{lma} \label{lma:almost}
  For any $\alpha > 0$ there exists
  $\varepsilon = \varepsilon(\alpha)$ 
  such that the following holds.
  If $G$ is an oriented graph on 
  $n$ vertices such that $\delta^0(G) \ge (7/18 - \varepsilon)n$,
  then either $G$ has a $TT_3$-tiling
  on all but at most $14$ vertices or $G$ is $\alpha$-extremal.
\end{lma}
\begin{proof}[Proof of Theorem~\ref{thm:near_tiling} and Lemma~\ref{lma:almost}]
  For the proof of Lemma~\ref{lma:almost}, let $\varepsilon < \alpha/50$.
%  and $n_0 = 2 n_0(\alpha)$ 
%  where $\varepsilon(\alpha)$ and $n_0(\alpha)$
%  are as in Proposition~\ref{clm:blowup}.
  For the proof of Theorem~\ref{thm:near_tiling}, let $\varepsilon = 0$.
  So, in either case, we have that $\delta^0(G) \ge (7/18 - \varepsilon)n$.

  Let $\s{M} = \s{T} \cup \s{P} \cup F \cup I$ 
  be a collection of vertex-disjoint subgraphs of $G$
  such that 
  every vertex in $G$ is contained in a subgraph of $\s{M}$,
  every $T \in \s{T}$ is a transitive triangle,
  every $P \in \s{P}$ is a directed path on $3$ vertices, 
  every $e \in F$ is an edge and 
  every $v \in I$ is a single vertex.
  Clearly such a set $\s{M}$ exists.
  Assume that $\s{M}$ is selected to maximize
  $(|\s{T}|, |\s{P}|, |F|)$ lexicographically.
  Let $X = V(\s{T})$, $Y = V(\s{P})$ and $Z = V(G) \setminus (X \cup Y)$.
  
  We will show that 
  if $\varepsilon = 0$, then $|\s{P}| \le 2$ and 
  if $\varepsilon > 0$ and $|\s{P}| \ge 4$, then 
  $G$ is $\alpha$-extremal. We will also show that $|F| \le 1$ and $|I| \le 3$, and this will prove the theorem.

\smallskip

  Let $B$ be a $(\s{P}, V(G))$ bipartite graph in which
  there is an edge between $P \in \s{P}$ and $v \in V(G)$ if and only
  if $G[P \cup v]$ contains a transitive triangle.
  By Proposition~\ref{prp:deg}, 
	$d_B(P) \ge 2(3\delta^0(G) - n) \ge (1/3 - 6 \eps) n $. 
  For every $P \in \s{P}$,
  by the maximality of $|\s{T}|$, $d_B(P, Y \cup Z) = 0$.
  Also by the maximality of $|\s{T}|$, 
  for every $T \in \s{T}$
  if there exists $x \in V(T)$ such that $d_B(x) \ge 2$,
  then $d_B(y) = 0$ for every $y \in V(T) - x$.
  Assume $|\s{P}| \ge 3$ and note that
  we then have that $e_B(\s{P}, V(T)) \le |\s{P}|$ for every $T \in \s{T}$. 
  Let 
  \begin{equation*}
    \s{T}' = \{ T \in \s{T} : e_B(\s{P}, V(T)) > 3\}.
  \end{equation*}
  We have that 
  %\footnote{changed U to V(G)} OK
  \begin{equation*}
    n + (|\s{P}| - 3)|\s{T}'| > 
    3|\s{T}| + (|\s{P}| - 3)|\s{T}'| \ge 
    e_B(\s{P}, V(G)) \ge
    (1/3 - 6\varepsilon)n|\s{P}|.
  \end{equation*}
  Which, since $|\s{T}'| < n/3$, 
  is a contradiction when $\varepsilon = 0$,
  so in this case, we must have that $|\s{P}| \le 2$.
  If $\varepsilon > 0$ and $|\s{P}| \ge 4$, then
  \begin{equation*}
    |\s{T}'| \ge 
    \left(\frac{|\s{P}|}{3(|\s{P}| - 3)} - \frac{1}{|\s{P}| - 3} -
    \frac{6|\s{P}|\varepsilon}{|\s{P}| - 3}\right)n \ge
    \left(\frac{1}{3} - 24\varepsilon\right)n.
  \end{equation*}

  For every $T \in \s{T}'$, since $e_B(\s{P}, T) \ge 4$,
  there exists $x_T \in V(T)$ such that $d_B(x_T) \ge 2$.
  Therefore, by the maximality of $|\s{T}|$,
  $d_B(x_T) = e_B(\s{P}, V(T) ) \ge 4$.
  Let 
  \begin{equation*}
    W = Y \cup Z \cup \bigcup_{T \in \s{T}'} \left(V(T) - x_T\right),
  \end{equation*}
  and note that $|W| > \left( 2/3 - 48 \varepsilon\right)n$.
  The graph $G[W]$ does not contain a transitive triangle.
  Indeed, if such a triangle $T$ exists
  and we define
  $\s{T}'' = \{T' \in \s{T}' : V(T) \cap V(T') \neq \emptyset\}$
  and $B' = B - \{P \in \s{P} : V(P) \cap V(T) \neq \emptyset\}$,
  then for every $T' \in \s{T}''$, we have that 
  \begin{equation*}
    d_{B'}(x_{T'}) \ge d_{B}(x_{T'}) - |Y \cap V(T)| \ge
    4 - |Y \cap V(T)| > |X \cap V(T)| \ge |\s{T}''|.
  \end{equation*} 
  Therefore, there is a matching covering $\s{T}''$ in $B'$.
  The edges in this matching
  correspond to $|\s{T}''|$
  disjoint transitive triangles in the graph induced
  by $(V(\s{T}'') \cup Y \cup Z) \setminus V(T)$,
  contradicting the maximality of $|\s{T}|$. 
  Hence, $G$ is $\alpha$-extremal.
%  Note that 
%  \begin{equation*}
%    \delta(G[W]) \ge \frac{\delta(G) - (n - |W|)}{|W|}|W|
%    > \frac{4/9 - 50\varepsilon}{2/3 - 48\varepsilon} |W| > 
%    (2/3 - 75\varepsilon)|W|,
%  \end{equation*}
%  so by Proposition~\ref{clm:blowup}, 

  Assume that there exist two distinct edges $ab$ and $cd$ in $F$.
  For any set $U \subseteq V(G)$, define 
  \begin{equation*}
    w(U) = d^+_G(a, U) + \sum_{v \in \{b,c,d\}}d_G(v, U).
  \end{equation*}
  Note that by the maximality of $|\s{P}|$,
  there are no triangles in $G[Z]$, 
  so for every $z \in Z$, $w(z) \le 2$.
  For any $P \in \s{P}$,
  the maximality of $|\s{T}|$ implies that there is no transitive triangle
  in the graph induced by $\{a,b, c, d\} \cup V(P)$.
  It is not hard to see that, with Proposition~\ref{prp:TT4}, 
  this gives us that $d^+_G(a, V(P)) + d_G(b, V(P)) \le 3$ and
  %\footnote{I changed 5 to 4 and 8 to 7, not as if it changes anything later.} 
  %OK
  $e_{G}(cd, V(P)) \le 4$,
  so $w(P) \le 7$.
  \begin{clm*}
    For every $T \in \s{T}$, $w(T) \le 8$.
  \end{clm*}
  \begin{proof}
    Assume that there exists $T \in \s{T}$ such that $w(T) \ge 9$.
    We will show that there exists a disjoint directed path
    on $3$ vertices and a transitive triangle in the
    graph induced by $\{a,b,c,d\} \cup V(T)$, contradicting
    the maximality of $|\s{P}|$.

    Remove the edges into $a$ from $G$ to form $G'$.
    Note that this implies that for every $x \in V(T)$,
    if $e_{G'}(ab, x) = 2$, then $abx$ is a transitive triangle.
    If, in addition to this, there exists $y \in V(T) - x$ such
    that $e_{G'}(cd, y) = 2$, then we have the desired directed
    path on $3$ vertices.
    This is the case when for one of the edges 
    $e \in \{ab, cd\}$, $d_{G'}(e, V(T)) = 5$.
    Indeed, if $f \in \{ab, cd\} - e$, then
    $e_{G'}(f, V(T)) \ge 4$, so
    we can pick $x \in V(T)$ such that $e_{G'}(f, x) = 2$.
    Since then $e_{G'}(e, V(T) - x) \ge 3$,
    we can pick $y \in V(T) - x$ such that $e_{G'}(e, y) \ge 2$.
    Therefore, we are only left to consider the cases when one of
    $ab$ or $cd$ and $V(T)$ induce a tournament on $5$ vertices in $G'$.

    If $e_{G'}(ab, V(T)) = 6$, then $e_{G'}(cd, V(T)) \ge 3$ and 
    for one of $c$ or $d$, say $c$, $e_{G'}(c, V(T)) \ge 2$, so if $x$ and $y$
    are the neighbors of $c$ in $V(T)$, $cxy$ is a triangle and therefore
    a directed path on $3$ vertices.
    If $z \in V(T) - x - y$, then $zab$ is a transitive triangle.
    
If $e_{G'}(cd, V(T)) = 6$, then $e_{G'}(ab, V(T)) \ge 3$.
    We can assume that  $T$ is the unique transitive triangle in 
    $G'[\{a,b\} \cup V(T)]$, because, if it is not,
    then the graph induced by the vertices of $V(T)$ not in this triangle 
    and $\{c,d\}$ would contain a triangle, which contains a directed path on $3$ vertices.
    This implies that $e_{G'}(ab, v) = 1$ for every $v \in V(T)$,
    and that $e_{G'}(a, V(T)) \le 1$.
    This further implies that $b$ has an outneighbor $x \in V(T)$,
    so $abx$ is a directed path on $3$ vertices.
    Since $G'[\{c,d\} \cup V(T) - x]$ is a tournament on $4$ vertices,
    it contains a transitive triangle 
    by Proposition~\ref{prp:TT4}. 
  \end{proof}
  Therefore,
  \begin{equation*}
    \left(\frac{49}{18} - 7\varepsilon\right)n \le 
    7\delta^0(G) \le 
    w(V(G)) \le 
    2|Z| + 7|Y|/3 + 8|X|/3 \le 8n/3,
  \end{equation*}
  a contradiction. Hence $|F| \le 1$.

  By the maximality of $|F|$,
  $I$ is an independent set.
  Since there are no triangles in $G[Z]$,
  $e_G(I, e) \le |I|$ for every $e \in F$.
  Let $T \in \s{P} \cup \s{T}$. 
  If $e_G(I, V(T)) > 2|I| + 1$, then 
  there exist vertices $v_1, v_2 \in I$ such that
  $e_G(v_1, V(T)) = e_G(v_2, V(T)) = 3$.
  Furthermore, in the graph induced by $\{v_1, v_2\} \cup V(T)$,
  if $T \in \s{P}$, then there is a triangle and a disjoint edge, and 
  if $T \in \s{T}$, then there is a transitive triangle and a disjoint edge. 
  Since both cases violate the maximality of $|F|$,
  we have that
  \begin{multline*}
    |I|\left(\frac{7}{9} - 2\varepsilon\right)n \le 
    |I|\delta(G) \le 
    e_G(I, V(G) \setminus I) \\
    \le |I||F| + (2|I| + 1)|\s{P} \cup \s{T}| \le 
    \frac{|I||Z|}{2} + \frac{(2|I| + 1)|X \cup Y|}{3} 
    \le |I|\frac{2n}{3} + \frac{n}{3}.
  \end{multline*}
  Hence, $|I| \le 3 + 18 \varepsilon$.
\end{proof}

Using Lemmas~\ref{lma:absorption} and \ref{lma:almost}, we can prove  Lemma~\ref{lma:nonextremal}.
%\footnote{It should say Lemma 1.3, but did not manage to write something properly...} OK

\begin{proof}[Proof of Lemma~\ref{lma:nonextremal}]
  Let $\eps = \min\{\eps(\alpha/3)/2, 1/250\}$ where
  $\eps(\alpha/3)$ is as in Lemma~\ref{lma:almost} and
  let $\sigma_0 = \sigma_0(\eps)$ be as in Lemma~\ref{lma:absorption}.
  Assume that $n$ is sufficiently large.
  So, by Lemma~\ref{lma:absorption}, 
  there exists $\sigma < \min\{\sigma_0, \eps/3, \alpha/3 \}$ for which 
  there exists $U \subseteq V(G)$ such that 
  $|U| \le 3 \sigma n < \varepsilon n$ 
  and the conclusion of Lemma~\ref{lma:absorption} holds.
  As $G[W \cup U]$ has a perfect $TT3$-tiling when $W = \emptyset$,
  we can conclude that $|U|$ is divisible by $3$.
  Let $G' = G - U$ and $n' = |V(G')|$ and note that 
  $\delta^0(G') \ge (7/18 - 2 \varepsilon)n \ge (7/18 - 2\varepsilon)n'$
  and $n'$ is divisible by $3$.
  Assume that $G$ is not $\alpha$-extremal.
  Because $|U| \le 3 \sigma n < \alpha n$, we have that
  \begin{equation*}
    (2/3 - \alpha/3)n' = (2/3 - \alpha/3)(n - |U|) > (2/3 -  \alpha)n, 
  \end{equation*}
  so $G'$ is not $(\alpha/3)$-extremal.
  Therefore, 
  by Lemma~\ref{lma:almost} 
  there exists a $TT_3$-tiling on all of $V(G')$ except a set $W$ of size at most $12$.
  Since $|W| < 3 \sigma^2 n$ there exists a perfect $TT_3$-tiling
  of $G[U \cup W]$ completing the proof.
\end{proof}

\section{The $\alpha$-extremal case}
\label{sec:extremal}

In this section we prove Lemma~\ref{lma:extremal}.
We start with some well-known and simple propositions
regarding matchings in graphs.
\begin{prp}
  \label{prp:deg_implies_matching}
  Every graph $G$ on $n$ vertices has a matching of size
  at least $\min\{\floor{n/2}, \delta(G)\}$.
\end{prp}
\begin{proof}
  Let $M$ be a maximum matching in $G$ and
  assume $|M| < \min\{ \floor{n/2}, \delta(G) \}$.
  Let $U$ be the set of vertices that are incident to an edge in $M$. 
  Because $|M| \le n/2 - 1$,
  there exist distinct $x,y \in V(G) \setminus U$.
  Since $M$ is a maximum matching, $e_G(\{x,y\}, V(G) \setminus U) = 0$ which implies 
  \begin{equation*}
    e_G(\{x,y\}, U) \ge 2 \delta(G) > 2|M|.
  \end{equation*}
  So there exists $e \in M$ such that $e_G(\{x,y\}, e) \ge 3$.
  This contradicts the maximality of $M$.
\end{proof}
\begin{prp}
  \label{prp:halls_matching}
  Let $G$ be an $(X,Y)$-bipartite graph with $d_G(x) \ge a$ for every 
  $x \in X$ and $d_G(y) \ge b$ for every $y \in Y$.
  If $|X| = |Y|$ and $a + b \ge |X|$, then $G$ contains a perfect matching.
\end{prp}
\begin{proof}
  We show that $G$ satisfies Hall's condition. Let
  $X' \subseteq X$ be non-empty, let $x \in X'$
  and let $Y'$ be the set of vertices in $Y$ that are adjacent to a vertex in $X'$.
  Clearly, $|Y'| \ge d_G(x) \ge a$,
  so assume $|X'| > a$. 
  We have that $d_G(y) \ge b > |X \setminus X'|$ for every $y \in Y$. 
  Hence, $y \in Y'$ and $|Y'| = |Y| \ge |X'|$.
\end{proof}

Let $G$ be a $(V_1,V_2)$-bipartite graph.  For
$X_1 \subseteq V_1$ and $X_2 \subseteq V_2$ both non-empty,
define $d_G(X_1, X_2) := \frac{e_G(X_1, X_2)}{|X_1||X_2|}$
to be the \emph{density} of $G$.
For constants $0 < \varepsilon, d < 1$,
we say that $G$ is \emph{$(d, \varepsilon)$-regular} if 
$$(1 - \varepsilon)d \le d_G(X_1, X_2) \le (1 + \varepsilon)d$$
whenever $|X_i| \ge \varepsilon |V_i|$ for $i = 1, 2$.
We say that $G$ is \emph{$(d, \varepsilon)$-superregular} if
$G$ is $(d, \varepsilon)$-regular and
$(1 - \varepsilon)d|V_i| \leq d_G(v, V_i) \leq (1+\varepsilon)d|V_i|$ for
every $v \in V_{3-i}$ and $i \in \{ 1,2 \}$. 
%Note that the preceding definitions
%match the corresponding definitions in \cite{KO-match}.
\begin{prp}
  \label{prp:super_reg}
  For any $0< \varepsilon <1$, if $G$ is a $(V_1, V_2)$-bipartite graph 
  such that $|V_1| = |V_2| = n$ and $\delta(G) \ge (1 - \varepsilon)n$ 
  then $G$ is $(1, \varepsilon^{1/2})$-superregular. 
\end{prp}
\begin{proof}
  It is clear that we only need to show that $G$ is $(1,\varepsilon^{1/2})$-regular.
  Let $X_i \subseteq V_i$ such that $|X_i| \ge \varepsilon^{1/2} n$ for
  $i \in \{1, 2\}$.
  We have that
  \begin{equation*}
    1 \ge d(X_1, X_2) \ge \frac{|X_1|(|X_2|-\varepsilon n)}{|X_1||X_2|}
    = 1 - \frac{\varepsilon n}{|X_2|} \ge 1 - \varepsilon^{1/2}. \qedhere
  \end{equation*}
\end{proof}

The following lemma follows immediately from 
the Chernoff type bounds on the hypergeometric
distribution (see Theorem 2.10 in \cite{luczak}).
\begin{lma}
  \label{lma:hoeffding}
  For every $0< \eta < 1$ 
  there exists $k = k(\eta) > 0$ such
  when $V$ is a set, $X \subseteq V$
  and $m$ is  a positive integer such that 
  $m \le |V|$ the following holds.
  If $U$ is selected uniformly at random from $\binom{V}{m}$, 
  then with probability at least $1 - e^{-km}$
  \begin{equation*}
  \frac{  |X|}{|V|} - \eta \le 
    \frac{|X \cap U|}{|U|} \le 
    \frac{|X|}{|V|} + \eta.
  \end{equation*}
\end{lma}

A partition of a set is \emph{equitable} if any two parts differ in size by
at most $1$.
\begin{prp}
  \label{prp:split}
  For every $\eta > 0$ there exist integers
  $k = k(\eta) > 0$ and $n_0 = n_0(\eta)$ such that 
  when $F$ is an $(A, B)$-bipartite graph with
  $|A| = |B| = n$ for $n \ge n_0$
  the following holds.
  If an equitable partition $\{A_1, A_2\}$ of $A$ and an 
  equitable partition $\{B_1, B_2\}$ of $B$ are both 
  chosen uniformly at random from all such partitions,
  then with probability at least $1 - e^{- k n}$ we have 
  \begin{equation*}
    d_F(A, B) - \eta \le d_F(A_i, B_j) \le  d_F(A, B) + \eta
  \end{equation*}
  for every $1 \le i, j \le 2$.
\end{prp}
\begin{proof}
  Choose partitions $\{A_1, A_2\}$ and
  $\{B_1, B_2\}$ as in the proposition and let $k = 5 k(\eta/2)$,
  where $k(\eta/2)$ is as in in Theorem~\ref{lma:hoeffding}, and assume that $n$ is sufficiently large.
  Let $1 \le i, j \le 2$.
  By Theorem~\ref{lma:hoeffding}, 
  for any $v \in A$ with probability at least 
  $1 - e^{- 5k |B_j|} \ge 1 - e^{- 2k n}$
  \begin{equation}
    \label{eq:density_inequalities}
   \frac{ d_F(v, B)}{|B|} - \frac{\eta}{2} \le \frac{d_F(v, B_j)}{|B_j|} \le \frac{d_F(v, B)}{|B|} + \frac{\eta}{2},
  \end{equation}
  and the analogous statement holds for every $v \in B$.
  So with probability at least 
  \begin{equation*}
    1 - 2ne^{-2kn} \ge 1 - e^{-kn} 
  \end{equation*}
   \eqref{eq:density_inequalities} holds for every $v \in V(G)$.  
  Therefore,
  \begin{equation*}
    \begin{split}
      d_F(A_i, B_j) &= 
      \frac{\sum_{v \in A_i}d_F(v, B_j)}{|A_i||B_j|} \ge 
      \frac{\sum_{v \in A_i}d_F(v, B)}{|A_i||B|} - \frac{\eta}{2} =
      \frac{\sum_{v \in B}d_F(v, A_i)}{|A_i||B|} - \frac{\eta}{2}  \\
      & \ge \frac{\sum_{v \in B}d_F(v, A)}{|A||B|} - \eta
      = d_F(A, B) - \eta.
    \end{split} 
  \end{equation*}
  By a similar computation, the upper bound also holds.
\end{proof}
\begin{thm}[K\"{u}hn and Osthus  \cite{KO-match}]
  \label{thm:matching}
  For all positive constants $ d, \xi_0 , \eta \leq 1$ there is a positive $\varepsilon =
  \varepsilon(d, \xi_0 , \eta)$ and an integer $n_0 = n_0(d, \xi_0 , \eta)$ such that the following holds for all $n \geq n_0$
  and all $\xi\geq \xi_0$. Let $G$ be a $(d, \varepsilon)$-superregular bipartite graph whose vertex
  classes both have size $n$ and let $F$ be a subgraph of $G$ with $e(F) = \xi e(G)$. Choose a
  perfect matching $M$ uniformly at random in $G$. Then with probability at
  least $1 -e^{-\varepsilon n}$
  we have 
  \begin{equation*} 
    \xi - \eta \leq \frac{|M \cap E(F)|}{|M|} \leq \xi + \eta.
  \end{equation*}
\end{thm}

\begin{prp}
  \label{prp:cyctri}
  Let $G$ be  an oriented graph and let
  $x \in V(G)$ and let $a,b,c \in N_G(x)$.
  If  $abc$ is a cyclic triangle in~$G$,
  then $xe$ is a transitive triangle
  for at least two edges $e \in \{ab, bc, ca\}$.
\end{prp}
\begin{proof}
  Let $i = d^+_G(x, \{a, b, c\})$.
  By symmetry, there are four cases depending on the value of $i$.
  Furthermore, 
  by reversing the edges of $G$ it is easy to see that the
  cases when $i=j$ are equivalent to the cases when $i = 3 - j$.
  It is easy to verify the statement when $i = 3$ and  when $i = 2$, we omit the details.
\end{proof}

We will use following lemma when finishing the proof of Lemma~\ref{lma:extremal}.
Lemma~\ref{lma:finish_ext} essentially states that if
a graph looks very similar to the graph depicted in Figure~\ref{fig:extremal}
except that $|W_1| + |W_2| + |W_3| = 2(|U_1| + |U_2|)$, 
then there exists a perfect $TT_3$-tiling of $G$.
%As a final step, we show that
%if $\delta(G) \ge 7n/18$ and $G$ is $\alpha$-extremal, then
%we can remove a small number of vertex-disjoint transitive triangles to leave
%a graph that looks like Figure~\ref{fig:extremal} with
%the property that $|W_1| + |W_2| + |W_3| = 2(|U_1| + |U_2|)$.
%We say that $G$ is a \emph{blow-up of a cyclic triangle} if $G$ is an oriented
%graph on $n$ vertices such that there exists a vertex partition $W_1,W_2,W_3$
%with $|W_i | = \lfloor (n+i-1)/3\rfloor$ for $i \le 3$ and every vertex
%in $W_i$ contains $W_{i+1}$ in its out-neighborhood.
%
%For a vertex set $V$ such that $|V| = 18m$ for some $m \in \N$, 
%let $\s{C}(V)$ be the collection of oriented graphs $C$
%on $V$ for which there is a partition $\{U,W\}$ of $V$ such that
%$|U| = 6m$, $|W| = 12m$, 
%$U$ is an independent set, 
%$C[W]$ is  a blow-up of a cyclic triangle
%and $d^+_C(w,U)=d^-_C(w,U) = 3m$ for every $w \in W$.
%\begin{lma}
%  \label{lma:finish_ext}
%  There exists a constant $\beta >0$ and an integer $n_0$ such that for
%  any $m \in \N$ where $n = 18m \ge n_0$ 
%  and any oriented graph $G$ on $n$ vertices the following holds.
%  If there exists $C \in \s{C}(V(G))$ such that
%  $\Delta(C - G) \le \beta m$,
%  then $G$ contains a perfect $TT_3$-tiling.
%\end{lma}
\begin{lma}
  \label{lma:finish_ext}
  There exists a constant $\beta >0$ and an integer $n_0$ such that for
  any $m \in \N$ where $n = 18m \ge n_0$ 
  and any oriented graph $G$ on $n$ vertices the following holds.
  If there exists a partition $\{W_1, W_2, W_3, U\}$ of
  $V(G)$ such that 
  $|W_1| = |W_2| = |W_3| = 4m$, $|U| = 6m$
  and for every $i \in [3]$ and $w \in W_i$,
  \begin{equation}
    \label{eq:lma_deg_cond_W_to_W}
    d^+(w, W_{i+1}), d^-(w, W_{i-1}) \ge (4 - \beta) m \text{ and }
  \end{equation}
  \begin{equation}
    \label{eq:lma_deg_cond_W_to_U}
    d^+(w, U), d^-(w, U) \ge (3 - \beta) m;
  \end{equation}
  and, for every $u \in U$ and $i \in [3]$,
  \begin{equation}
    \label{eq:lma_deg_cond_U_to_W}
    d(u, W_i) \ge (4 - \beta)m,
  \end{equation}
  then $G$ contains a perfect $TT_3$-tiling.
\end{lma}
\begin{proof}
  Let $\eta = 1/12$, 
  $\varepsilon = \varepsilon(1, \eta/2, \eta/2)$,
  $\beta = \min\{ \varepsilon^2, 1/24\}$ 
  where $\varepsilon(1, \eta/2, \eta/2)$ is as in Theorem~\ref{thm:matching}.
  Assume $m$ is sufficiently large.
  %\footnote{I added a /2 factor} TM OK
%  Let $\{U,W\}$ be a partition of $V(G)$ for which
%  $|W| = 12m$ and $C[W]$ is a blow-up of a cyclic triangle
%  with parts $W_1$, $W_2$ and $W_3$ each of order $4m$
%  such that every vertex in $W_i$ has $W_{i+1}$ in its out-neighborhood
%  and $d^+_C(w, U) = d^-_C(w, U) = 3m$ for every $w \in W$.
  Let $W = W_1 \cup W_2 \cup W_3$. 
%  and note that,  
%  by possibly deleting edges in $G[W]$, we
%  we may assume that $G$ still satisfies, $there are no transitive triangles in $G[W]$.

  Let $F = E(G[W_1, W_2]\cup G[W_2, W_3]  \cup G[W_3, W_1])$ and define a bipartite graph $B$ with classes $U,F$ as follows.
A vertex $u\in U$ and an edge $xy\in F$ form an edge in $B$ if $uxy$ is a transitive triangle in $G$.
%\footnote{I added this definition, please check} Looks OK to me.
  Clearly $|F| \le 3(4m)^2 = 48m^2$.

  \begin{clm*}
    For every $u \in U$, $d_B(u) \ge (2/3 - \beta)48m^2$.
  \end{clm*}
  \begin{proof}
    Let $P(u)$ be the set of pairs of the form 
    $(e, abc)$ where 
    $a \in W_1 \cap N_G(u)$, 
    $b \in W_2 \cap N_G(u)$ and 
    $c \in W_3 \cap N_G(u)$,
    $abc$ is a cyclic triangle 
    and $e \in \{ab, bc, ca\} \cap N_B(u)$.
    By Proposition~\ref{prp:cyctri},
    for every $(e, abc) \in P(u)$
    the cyclic triangle $abc$ appears at least twice as the second element 
    of a pair in $P(u)$.
    Therefore, by \eqref{eq:lma_deg_cond_W_to_W} and \eqref{eq:lma_deg_cond_U_to_W},
    \begin{equation*}
      |P(u)| \ge 
      2\cdot(4 - \beta)m\cdot(4 - 2\beta)m\cdot(4 - 3\beta)m
      > (1 - 3\beta/2)128m^3.
   \end{equation*}
    Since any edge can appear as the first element
    in at most $4m$ of the pairs in $P(u)$,
    \begin{equation*}
      d_B(u) \ge 
      |P(u)|/(4m) \ge (1 - 3\beta/2)32m^2
      = (2/3 - \beta)48m^2. \qedhere
    \end{equation*}
  \end{proof}

  For every $u \in U$, let $F(u)$ be the graph on $W$
  with edge set
  %\footnote{Do we need $\cap F$, it should contain $N_B(u)$}  No, I don't
  %think so
  $N_B(u)$.
  By Proposition~\ref{prp:split} and the union bound,
  there exists an 
  equitable partition $\{W^1_i, W^2_i\}$ of $W_i$ 
  for each $i \in [3]$,
  such that for every $u \in U$,
  \begin{equation*}
    d_{F(u)}(W^k_i, W^{\ell}_j) \ge d_{F(u)}(W_i, W_j) - \eta/2
  \end{equation*}
  for all $1 \le k,\ell \le 2$ and $j \in \{1,2,3\} - i$.
  Let $G_1 = G[W^1_1, W^1_2]$,
  $G_2 = G[W^2_2, W^2_3]$ and
  $G_3 = G[W^1_3, W^2_1]$.
  Note that $\delta(G_i) \ge (2 - \beta)m$ for every
  $i \in [3]$, so $G_i$
  is $(1, \beta^{1/2})$-superregular
  %\footnote{removed some epsilon} 
  by Proposition~\ref{prp:super_reg}.
  Therefore, by Theorem~\ref{thm:matching} and the union bound, 
  there exists a perfect matching
  $M_i$ of $G_i$ such that
  \begin{equation*}
\frac{    |M_i \cap E(F(u))|}{2m}
    \ge \frac{|E(G_i) \cap E(F(u))|}{e(G_i)} - \frac{\eta}{2} 
    \ge d_{F(u)}(W_i, W_{i+1}) - \eta
  \end{equation*}
  for every $u \in U$ and $i \in [3]$.
  Note that, because $\beta^{1/2} \le \varepsilon(1, \eta/2, \eta/2)$,
  Theorem~\ref{thm:matching} may not apply when
  $|E(G_i) \cap E(F(u))|/e(G_i) \le \eta/2$, but in that case
  the inequality is vacuously true.
  Observe that  $M = M_1 \cup M_2 \cup M_3$
  is a perfect matching of $G[W]$ and 
  for $B' = B[U, M]$, and every
  %\footnote{I changed many things below:}  OK
  $u \in U$
  \begin{equation*}
    \begin{split}
    \frac{  d_{B'}(u)}{|M|} &=
     \frac{ |M \cap E(F(u))|}{6m}       = \frac{1}{3}\frac{\sum_{i = 1}^3 |M_i \cap E(F(u))|}{2m} \\
     &\ge \frac{1}{3} {\sum_{i=1}^3 \left(d_{F(u)}(W_i, W_{i+1})-\eta\right)}
      = \frac{d_B(u)}{48m^2}  - \eta \ge
     \frac{ 2}{3} - (\beta + \eta).
    \end{split}
  \end{equation*}
  We also have that, by \eqref{eq:lma_deg_cond_W_to_U},
  for every $xy \in M$,
  \begin{equation*}
    d_{B'}(xy) \ge |N^+_G(x, U) \cap N_G(y, U)| \ge (3 - 3\beta)m > (1/2 - \beta)6m.
  \end{equation*}
  Note that since 
  \begin{equation*}
    2/3 - (\beta + \eta) + 1/2 - \beta \ge 7/6 - 2\beta - \eta \ge 1,
  \end{equation*}
  Proposition~\ref{prp:halls_matching} implies
  that $B'$ has a perfect matching.
  This perfect matching corresponds to a perfect $TT_3$-tiling of $G$.
\end{proof}

%  For every $x \in V$, $v \in A$, $i \in \{1,2,3\}$ and $j \in \{1,2\}$
%  the probability that 
%  $|N_B(x) \cap A^1_i| \ge |N_B(x) \cap A_1|/2 - \varepsilon m$ 
%  is greater than $1 - e^{-km}$ for some $k$.
%  Therefore there exist partitions that satisfy all of the inequalities 
%  simultaneously.
%  Consider the three disjoint bipartite graphs, 
%  $G_1' = G_1[A^1_1 \cup A^2_2]$,
%  $G_2' = G_2[A^1_2 \cup A^3_2]$
%  and
%  $G_3' = G_3[A^1_3 \cup A^1_2]$.
%  Note that each is for any $x \in A$ if $x \in G_i$ then
%  $d_{G_i}(x) \ge (1 - \varepsilon)2m$.
%  Therefore, with $\sigma = \varepsilon^{1/2}$,
%  we have that $G_i$ is $(\sigma, \varepsilon)$-super regular.
%  By Theorem~\ref{thm:matching}, there exists a perfect matching $M_i$
%  of $G_i$ such that
%  \begin{equation*}
%    d_B(x, M_i) \ge 
%    (1 - \varepsilon) \frac{d_B(x, E(G_i'))}{4m^2}  2m \ge
%    (1 - 2\varepsilon) \frac{d_B(x, E(G_i))}{16m^2} 2m.
%  \end{equation*}
%  Let $B' = B[B \cup M_1 \cup M_2 \cup M_3]$.
%  For every $x \in B$, by the Claim
%  \begin{equation*}
%    \begin{split}
%    d_{B'}(x) \ge \sum_{i=1}^{3} d_{B'}(x, M_i) &\ge
%    \sum_{i=1}^{3} (1 - 2\varepsilon) \frac{d_B(x, E(G_i))}{16m^2} 2m \\
%    &= (1 - 2\varepsilon) \frac{d_B(x)}{48m^2} 6m
%    \ge (2/3 - 3\varepsilon)6m.
%  \end{split}
%  \end{equation*}

\begin{proof}[Proof of Lemma~\ref{lma:extremal}]
  %\footnote{should be latexed properly}
  Let $\beta$ be as in Lemma~\ref{lma:finish_ext}.
  Let $\tau = \beta/320$ and let $\alpha = \tau^3$.

  Let $\gamma > 0$ and 
  let $\s{W} = \{W_1, W_2, W_3\}$ be a collection of 
  three disjoint vertex subsets.
  We say that $v \in V(G)$ is \emph{$(i, \gamma)$-cyclic} for the
  triple $\s{W}$ if 
  \begin{equation*}
    d^+_G(v, W_{i-1}) + d_G(v, W_i) + d^-_G(v, W_{i+1}) \le \gamma n,
  \end{equation*}
  and that $v$ is \emph{$\gamma$-cyclic} for $\s{W}$ if
  $v$ is $(i, \gamma)$-cyclic for some $i$.
  The triple $\s{W}$ is $\gamma$-\emph{cyclic} if every vertex
  in $W_i$ is $(i, \gamma)$-cyclic for every $i \in [3]$.
  A vertex is \emph{$\gamma$-bad} for $\s{W}$ if it is not $\gamma$-cyclic. 
  The following claim follows from the preceding definition.
  \setcounter{clm}{0} 
  \begin{clm}
    \label{clm:bad_vertex}
    For any $1 > \gamma > \gamma' \ge 0$, if
    a vertex $v$ is $\gamma$-bad for $\{W_1, W_2, W_3\}$ and
    $|X| \le \gamma' n$, then 
    $v$ is $(\gamma - \gamma')$-bad for 
    $\{W_1 \setminus X, W_2 \setminus X, W_3 \setminus X\}$.
  \end{clm}
  For any $\lambda$,
  we say that $\s{W}$ is \emph{$\lambda$-equitable} if 
  \begin{equation*}
    ||W_i| - |W_j|| \le \lambda n \text{ for every $i,j \in [3]$}
  \end{equation*}
  and $|V(\s{W})| \ge (2/3 - \lambda)n$.
  Note that this implies that
  $|W_i| \ge (2/9 - \lambda)n$ for every $i \in [3]$.
%  Say that $\s{W}$ is $\gamma$-{\it cyclic}, if it is 
%  $\gamma$-equitable and for every $i \in [3]$
%  every vertex in $W_i$ is $(i, \gamma)$-cyclic.

  Let $\s{W} = \{W_1, W_2, W_3\}$ be a $\lambda$-equitable triple 
  and let $v \in V(G)$
  be $(i, \gamma)$-cyclic for $\s{W}$.
  By the degree condition,
  \begin{equation*}
    \begin{split}
      d^-_G(v, W_{i-1}) + d^+_G(v, W_{i+1}) &=
      d_G(v, V(\s{W})) - 
      \left( d^+_G(v, W_{i-1}) + d_G(v, W_i) + d^-_G(v, W_{i+1})\right)\\
      & \ge |V(\s{W})| - 2n/9 - \gamma n.
    \end{split}
  \end{equation*}
  Therefore, since $|W_1|,|W_2|,|W_3| \ge (2/9 - \lambda)n$, 
  we have the following:
  \begin{equation}
    \label{eq:good_degs}
    \begin{split}
      d^-_G(v, W_{i-1}) &\ge 
      |W_{i-1}| + |W_i| - 2n/9 - \gamma n \ge
      |W_{i-1}| - (\gamma + \lambda) n, \\
      d^+_G(v, W_{i+1}) &\ge |W_{i+1}| - (\gamma + \lambda) n \text{ and }   \\
      d^-_G(v, W_{i-1}), d^+_G(v, W_{i+1}) &\ge (2/9 - 2\lambda - \gamma)n.
    \end{split}
  \end{equation}

  \begin{clm}
    \label{clm:tt3_exists}
    Let $0<\gamma <1/27$ and let 
    $\s{W} = \{W_1, W_2, W_3\}$ be such that
    $\s{W}$ is both $\gamma$-cyclic and $\gamma$-equitable.
    If $v \in V(G)$ such that
    there are no transitive triangles in $G[\{v\} \cup W]$ that
    contain~$v$, then $v$ is $0$-cyclic for $\s{W}$.
  \end{clm}
  \begin{proof}
    Since $|V(\s{W} ) | \ge (2/3 - \gamma)n > 11n/18$, 
    there exists an $x \in N^+_G(v, W_{i+1})$ for some $i \in [3]$.
    Let $I_x = N^-_G(x, W_i)$.
    By \eqref{eq:good_degs},
    $|I_x| \ge (2/9 - 3\gamma)n$.
    Suppose that $v$ is not $(i, 0)$-cyclic, i.e.\ there exists 
    \begin{equation*}
      y \in N^+_G(v, W_{i-1} \cup W_i) \cup N^-_G(v, W_i \cup W_{i+1}).
    \end{equation*}
    If $y \in N^+_G(v, W_{i-1} \cup W_{i})$, then let
    $I_y = N^-_G(y, W_{i+1} \cup W_{i-1})$ and
    if $y \in N^-_G(v, W_{i} \cup W_{i+1})$, then let
    $I_y = N^+_G(y, W_{i+1} \cup W_{i-1})$.
    Again by \eqref{eq:good_degs}, we have that
    $|I_y| \ge (2/9 - 3 \gamma)n$.
    Note that $v$ has no neighbors in $I_x \cup I_y$, because any 
    such neighbor would imply a 
    transitive triangle containing $v$ in $G[\{v\} \cup W]$.
    Also note that $I_x \subseteq W_i$ and 
    $I_y \subseteq W_{i+1} \cup W_{i-1}$, so
    $I_x$ and $I_y$ are disjoint, and
%    Note there does not exist $z \in N_G(v, I_x \cup I_y)$, because 
%    any such $z$ would imply that either
%    $vxz$ or $vyz$ is a transitive triangle.
%    Therefore, 
    \begin{equation*}
      |W| + \delta(G) - n \le d_G(v, W) \le |W| - |I_x| - |I_y| \le 
      |W| - (4/9 - 6 \gamma)n < |W| - 2n/9,
    \end{equation*}
    a contradiction.
  \end{proof}

  Recall, since $G$ is $\alpha$-extremal there exists $W \subseteq V(G)$ 
  such that $|W| \ge (2/3 - \alpha)n$ and $G[W]$ does not
  contain any transitive triangles.
  \begin{clm} \label{clm:blowup}
    There exists a $0$-cyclic partition $\s{W} = \{W_1, W_2, W_3\}$ of $W$
    such that for every $i \in [3]$
    \begin{equation*}
      (2/9 - \alpha)n \le |W_i| \le 2n/9.
    \end{equation*}
  \end{clm}
  \begin{proof}
%  \footnote{The proof was about $G$ not $G[W]$, I corrected it, please check} OK
    Let $G' = G[W]$ and
    note that
    \begin{equation}
      \label{eq:blowup_deg}
      \delta(G') \ge \delta(G) + |W| - n \ge |W| - 2n/9.
    \end{equation}
    Since $G'$ is $TT_3$-free, 
    for every $v \in W$ the sets $N^+_{G'}(v)$ and $N^-_{G'}(v)$
    are independent.
    This with \eqref{eq:blowup_deg} implies
    that both sets are of order at most $2n/9$ 
    and hence that 
    \begin{equation}
      \label{eq:blowup_semideg}
      \delta^0(G') \ge \delta(G') - 2n/9 \ge |W| - 4n/9.
    \end{equation}

    Since $G'$ is $TT_3$-free 
    there exists a cyclic triangle $w_1w_2w_3$ in $G'$.
    This also implies that, for any $i \in [3]$,
    the set $\widetilde{W}_i = N^+_{G'}(w_{i-1}) \cup N^-_{G'}(w_{i+1})$
    is disjoint from $N_{G'}(w_i)$.
    Hence, by \eqref{eq:blowup_deg}, $|\widetilde{W}_i| \le 2n/9$.
    %Again, because $G[W]$ is $TT_3$-free 
    Define $\widehat{W}_i = N^+_{G'}(w_{i-1}) \cap N^-_{G'}(w_{i+1})$.
    Then we have that $\widehat{W}_i$ is an independent set and,
    by \eqref{eq:blowup_semideg},
    \begin{equation*}
      2n/9 \ge 
      |\widehat{W}_i| \ge 
      d^+_{G'}(w_{i-1}) + d^-_{G'}(w_{i+1}) - |\widetilde{W}_i| \ge 
      2|W| - 10n/9 \ge (2/9 - 2\alpha)n.
    \end{equation*}
    Note that 
    \begin{equation*}
      \overrightarrow{E}_{G'}(\widehat{W}_{i-1}, \widehat{W}_{i+1}) \subseteq 
      \overrightarrow{E}_{G'}(N^-_{G'}(w_i), N^+_{G'}(w_i)) = \emptyset.
    \end{equation*}
    This gives us that
    $\widehat{\s{W}} = (\widehat{W}_1, \widehat{W}_2, \widehat{W}_3)$
    is $0$-cyclic.

    Let $X = W \setminus V(\widehat{\s{W}})$.
    By repeatedly applying Claim~\ref{clm:tt3_exists}, we can
    iteratively add each $x \in X$ to a set $\widehat{W}_i$ 
    for which $x$ is $(i, 0)$-cyclic for $\widehat{\s{W}}$.
    Let $\s{W} = \{W_1, W_2, W_3\}$ be the resulting collection.
    For every $i \in [3]$, the set $W_i$ is independent, so $|W_i| \le 2n/9$ by~\eqref{eq:blowup_deg} and moreover,
    because $|W| \ge (2/3 - \alpha)n$, 
    $|W_i| \ge (2/9 - \alpha)n$,
    So $\s{W}$ is the desired partition of $W$.
  \end{proof}

  Let $U = V(G) \setminus W$.
  If $v \in V(G)$ is $(i, \gamma)$-cyclic for $\s{W}$, then
  \begin{equation}
    \begin{split}
    \label{eq:in_out_good_to_U}
    d^+_G(v, U),d^-_G(v,U) &\ge 
    \delta^0(G) - \max\{|W_{i+1}|, |W_{i-1}|\} - \gamma n \\
    &\ge (1/6 - \gamma)n \ge |U|/2 - (\alpha/2 + \gamma)n.
  \end{split}
  \end{equation}
  We also have that, 
 % \begin{equation}
  %  \label{eq:good_to_U}
  %  \begin{split}
  %    d_G(v, U) &\ge 
  %    \delta(G) - (|W \setminus W_i| + d_G(v, W_i))  \\
  %    & \ge 7n/9 + (2/9 - \alpha)n - |W| - \gamma n = 
  %    |U| - (\alpha + \gamma)n.
  %  \end{split}
 % \end{equation}
 \begin{equation}
    \label{eq:good_to_U}    
      d_G(v, U) \ge 
      \delta(G) - (|W \setminus W_i| + d_G(v, W_i))  \\
       \ge {7n}/9 - 4n/9 - \gamma n = 
      |U| - (\alpha + \gamma)n.    
  \end{equation}

  By Claim~\ref{clm:blowup}, we
  can apply \eqref{eq:good_to_U} with $\gamma = 0$ to give us that
  %\footnote{should not $\alpha$ be $\alpha +\gamma$?} I think I clarified
  % this statement
  \begin{equation*}
%    \label{eq:UW_edges}
    e_G(W, U) \ge (|U| - \alpha n)|W| > |U||W| - \alpha n^2.
  \end{equation*}
  Defining
%  \begin{equation*}
  $
  Z = \{u \in U : d_G(u, W) < |W| - \tau n \},
  $
%  \end{equation*}
  we have that, since $\tau^3 = \alpha$,
  %\footnote{ using that $\alpha+\gamma< \rho^3$?} Yes, it
  % now state that explicitly.
  \begin{equation}
    \label{eq:size_of_Z}
    |Z| < \tau^{2} n.
  \end{equation}
  Let $Z(i)$ be the set of vertices in 
  $Z$ that are $(i, \tau)$-cyclic for $\s{W}$.
  Clearly $Z(1), Z(2)$ and $Z(3)$ are disjoint.
  Let $Z'' = \bigcup_{i=1}^3 Z(i)$,
  $W_i' = W_i \cup Z(i)$, 
  $\s{W}' = (W_1', W_2', W_3')$,
  $W' = V(\s{W}') = W \cup Z''$,
  $U' = U \setminus Z''$ and
  $Z' = Z \setminus Z''$.
  Note that, for every $i \in [3]$, 
  $(2/9 - \alpha)n \le |W_i'| \le 2n/9 + |Z|$ so
  $\s{W}'$ is $(2\tau^2)$-equitable and 
  that every vertex in $W_i'$ is $(i, \tau)$-cyclic for $\s{W}$.
  Since $|W' \setminus W| \le |Z|$,
  this implies that $\s{W'}$ is $(2\tau)$-cyclic.
  We also have that for every $u \in U' \setminus Z'$,
  \begin{equation}
    \label{eq:u_to_Wp}
    d_G(u, W') \ge |W| - \tau n \ge |W'| - |Z| - \tau n \ge |W'| - 2\tau n.
  \end{equation}
%    This furthermore implies that $\s{W}' = \{W_1', W_2', W_3'\}$
%    is $2\alpha^{1/3}$-good.

  We will now find three 
  collections $\s{T}_1$, $\s{T}_2$, $\s{T}_3$ of disjoint 
  transitive triangles.
  We define $X_i = V(\s{T}_i)$ and
  $Y_i = \bigcup_{j = 1}^{i} X_j$.
  The collections will be constructed so that 
  the sets $X_1, X_2$, and $X_3$ are disjoint.
  The collections will also have the following properties:
  \begin{enumerate}[(P\arabic*)]
    \item
      \label{prop:P1}
      $|W' \setminus Y_i| = 2|U' \setminus Y_i|$ for $i \in \{1,2,3\}$,
    \item
      \label{prop:P2}
      $Z' \subseteq  X_2$,
    \item
      \label{prop:P3}
      $|Y_3| \le \tau n$,
    \item
      \label{prop:P4}
      $|W_1' \setminus Y_3| = |W_2' \setminus Y_3| = |W_3' \setminus Y_3|$ and
    \item
      \label{prop:P5}
      $|V(G) \setminus Y_3|$ is divisible by $18$.
  \end{enumerate}
%  It will be clear that 
%  $|X_1| \le 3 |Z''|$,
%  $|X_2| \le 3 |Z'|$ and
%  $|X_3| \le 30 |Z|$,
%  so 
%  Therefore, 
%  with \eqref{eq:in_out_good_to_U},
%  $G - Y_3$ will satisfy the conditions of Lemm~\ref{lma:finish_ext} 
%  with $\beta = 5 \cdot 18 \tau$.
  Assume we have such collections.
  First, we will show that $G - Y_3$ has a $TT_3$-tiling
  by showing that $G - Y_3$ satisfies the
  conditions of Lemma~\ref{lma:finish_ext} with
  $\beta = 320 \tau = 16 \cdot 20 \tau$
  and 
  $\{W_1' \setminus Y_3, 
  W_2' \setminus Y_3, 
  W_3' \setminus Y_3, 
  U' \setminus Y_3
  \}$
  the required partition of $V(G - Y_3)$.
  To see this, first note that, by \ref{prop:P5}, there exist
  $m$ such that $|G - Y_3| = 18m$, and
  by \ref{prop:P1} and \ref{prop:P4},
  $|W_i' \setminus Y_3| = 4m$ for every $i \in [3]$
  and $|U' \setminus Y_3| = 6m$ and, 
  by \ref{prop:P3}, $m \ge n/20$.
  Since $\s{W'}$ is $(2\tau)$-cyclic and $(2\tau^2)$-equitable,
  \eqref{eq:good_degs} implies that
  for every $w \in W_i' \setminus Y_3$ and $i \in [3]$, 
  $w$ meets condition \eqref{eq:lma_deg_cond_W_to_W}.
  Furthermore, \eqref{eq:good_to_U} implies
  that $w$ also meets condition \eqref{eq:lma_deg_cond_W_to_U}.
  Finally, \ref{prop:P2} and \eqref{eq:u_to_Wp},
  imply that every $u \in U' \setminus Y_3$, 
  meets condition \eqref{eq:lma_deg_cond_U_to_W}.

  We begin the construction by finding a collection $\s{T}_1$ such that 
  $|W' \setminus X_1| = 2 |U' \setminus X_1|$.
  Call a transitive triangle $T$ \emph{standard} if
  $|V(T) \cap W'|=2$ and $|V(T) \cap U'|=1$. 
  Every transitive triangle $T \in \s{T}_2 \cup \s{T}_3$ will be standard 
  and this will give us Property~\ref{prop:P1}.

  Let $c = |W'| - 2n/3$
  and note that $-\alpha n \le c \le |Z''| \le \tau^2 n$, so $|c| \le \tau^2 n$.
%  \begin{equation}
%    \label{eq:size_of_c}
%    |c| \le \tau^{2} n.
%  \end{equation}
  Simple computations show that 
  the following claim gives the desired collection $\s{T}_1$.
  Indeed,
  if $c > 0$, then $|W'| - 3c =  2(n/3 - c) = 2|U'|$, and 
  if $c < 0$, then $|W'| - |c|  =  2(n/3 - |c|) = 2(|U'| - 2|c|)$.
  \begin{clm}
    There exists a collection $\s{T}_1$ of $|c|$ disjoint transitive triangles
    such that for every $T \in \s{T}_1$,
    \begin{itemize}
      \item
        if $c > 0$, $T \subseteq G[W']$; and 
      \item
        if $c < 0$, $|V(T) \cap W'| = 1$ 
        and $|V(T) \cap (U \setminus Z)| = 2$.
    \end{itemize}
  \end{clm}
  \begin{proof}
    First assume $c > 0$ and
    let $I = \{i \in [3] : |W_i'| > 2n/9\}$ 
    and $c_i = |W_i'| - \floor{2n/9}$ for $i \in I$. Note
    %\footnote{we need n upper bound on $c_i$} OK
    that, by Claim~\ref{clm:blowup}, $|W_i| \le 2n/9$, so 
    the fact that $|Z''| \le \tau^2 n$ implies that
    $c_i \le \tau^2 n$.
    For every $i \in I$, by the degree condition,
    we have that $\delta(G[W_i']) \ge c_i$.
    Therefore, Proposition~\ref{prp:deg_implies_matching} implies
    that there exists a matching $M_i$ of size $c_i$ in $G[W_i']$.
    %Since $\s{W'}$ is $\tau$-cyclic
    %\footnote{This actually was not defined for edges. Also, $d_B$ was undefined for edges in W.} 
    For every $xy \in M_i$,
    $x$ and $y$ are $(i, 2\tau)$-cyclic for $\s{W'}$
    and $\s{W'}$ is an $(2 \tau^2)$-equitable triple.  
    So, by \eqref{eq:good_degs}, 
    \begin{equation*}
      |N^-_G(x, W_{i-1}') \cap N^-_G(y, W_{i-1}')|
      \ge |W_{i-1}'| - 4(\tau + \tau^2)n,
    \end{equation*}
    and, 
    %\footnote{because the sizes might be smaller than 2n/9, a $-\alpha $ seems to be missing}, 
    similarly,  $|N^+_G(x, W_{i+1}') \cap N^+_G(y, W_{i+1}')| \ge 
    |W_{i+1}'| - 4(\tau + \tau^2)n$.
    Therefore, we can easily 
    match the edges $\bigcup_{i\in I} M_i$ to vertices 
    in $W'$ so that the matching corresponds
    to disjoint transitive triangles in $G$.
    Since $\sum_{i \in I} c_i \ge c$ we have
    the desired collection $\s{T}_1$.

    Now assume $c < 0$.
    Let $U'' = U \setminus Z = U' \setminus Z'$.
    By \eqref{eq:size_of_Z},
    we have that $|U''| \ge (1/3 - \tau^{2})n$ so 
    by the degree condition, $\delta(G[U'']) \ge (1/9 - \tau^{2})n$ 
    and there exists a matching $M$ of order
    %\footnote{here it is used that c is small} 
    $|c| \le \tau^2 n$ in $G[U'']$.
    By Proposition~\ref{prp:deg}, every $e \in E(G)$ has $n/6$ vertices 
    $v$ such that $ev$ is a transitive triangle.
    Therefore,  we can match 
    each edge $e \in M$ to a vertex $v_e \in  V(G) \setminus Z$ 
    so that the $ev_e$ are disjoint transitive triangles.
    Let $\s{T}_1'$ be this collection of transitive triangles.
    Suppose that there exists $T \in \s{T}_1'$
    such that $V(T) \subseteq U''$. 
    By Proposition~\ref{prp:TT4} and
    the fact that, by \eqref{eq:u_to_Wp},
%    the intersection of the neighborhoods
%    of the vertices in $T$ and $W'$ 
    $|\bigcap_{v \in V(T)}N_G(v, W')| \ge |W'| - 6\tau n$, 
    it is trivial to replace $T$ with a transitive triangle
    that has one vertex in $W'$ and an edge from $E(T)$
    and is also disjoint from $V(\s{T}_1' - T)$.
    By replacing every such triangle in this manner,  
    we can create the desired collection $\s{T}_1$.
  \end{proof}

  We now aim to find a collection $\s{T}_2$ of standard transitive
  triangles that satisfies Property~\ref{prop:P2}.
  Note that, by the definition of $Z'$, every vertex 
  in $Z'$ is $\tau$-bad for $\s{W}$ and hence
  is $\tau$-bad for $\s{W}'$.
  The following claim then follows from
  Claim~\ref{clm:bad_vertex} and Claim~\ref{clm:tt3_exists}.
  \begin{clm}
    There exists a collection $\s{T}_2$ of 
    $|Z' \setminus X_1|$ disjoint standard transitive triangles in $G - X_1$
    such that $|T \cap Z'| = 1$ for every $T \in \s{T}_2$.
%    such that for every $T \in \s{T}_2$ 
%    \begin{equation*}
%      |V(T) \cap Z'| = 1 \text{ and }
%      |V(T) \cap W'| = 2.
%    \end{equation*}
  \end{clm}
  \begin{proof}
%    By definition, every vertex in $Z'$ is 
%    This follows directly from Claim~\ref{clm:tt3_exists}.
    Let $\s{T}_2$ be a 
    collection of disjoint standard transitive triangles in $G - X_1$
    such that for every $T \in \s{T}_2$,
    $|V(T) \cap Z'| = 1$.
    Let $Y_2 = V(\s{T}_2) \cup X_1$.
    Suppose that $|\s{T}_2|$ is maximal among all such collections
    and that there exists $z \in Z' \setminus Y_2$.
%    Note that 
%    By the definition of $Z'$, $z$ is $\tau$-bad for
%    $\{W_1, W_2, W_3\}$ and hence 
%    $\tau$-bad for $\{W_1', W_2', W_3'\}$.
    Since $z$ is $\tau$-bad for $\s{W}'$,
    by Claim~\ref{clm:bad_vertex}
    and the fact that
    $|Y_2| < |X_1| + 3|Z'| < 6 \tau^2 n < \tau n$, $z$ is $0$-bad for 
    $\{
      W_1' \setminus Y_2,
      W_2' \setminus Y_2,
      W_3' \setminus Y_2
    \}$.
    Hence, by Claim~\ref{clm:tt3_exists}, 
    %Hence\footnote{I do not see this}, 
    there exists a transitive triangle containing $z$ and two vertices
    in $V(\s{W}') \setminus Y_2'$. 
    Adding $T$ to $\s{T}_2$ contradicts the maximality of $|\s{T}_2|$.
  \end{proof}
  Let $W_i'' = W_i' \setminus Y_2$ for every $i \in [3]$.
  Since $\s{W'}$ is $(2\tau^2)$-equitable and
  $|Y_2| \le |X_1| + 3 |Z'| \le 6 \tau^2$,
  the collection
  $\s{W}'' = \{W_1'', W_2'', W_3''\}$ is $(8\tau^{2})$-equitable.

  Because $|\s{T}_1 \cup \s{T}_2| \le 2|Z| \le 2\tau^2 n$, 
  if we we can find a collection $\s{T}_3$ of at most 
  $17 \tau^2 n \le \tau n/3 - 2\tau^2 n$
  disjoint standard transitive triangles in $G - Y_2$
  that satisfies \ref{prop:P4} and \ref{prop:P5},
  we will also satisfy Property~\ref{prop:P3}.
  This is quite easy to do, as we now describe.

  Let $\pi$ be a permutation of $[3]$ such that
  $|W_{\pi(1)}''| \le |W_{\pi(2)}''| \le |W_{\pi(3)}''|$.
  Let $M_1$, $M_2$ and $M_3$ be disjoint edge sets such that
  their union is a matching and
  \begin{itemize}
    \item
      $|M_1| = |W_{\pi(3)}''| - |W_{\pi(2)}''|$, 
      $|M_2| = |W_{\pi(3)}''| - |W_{\pi(1)}''|$,
    \item
      $M_1 \subseteq E_G(W_{\pi(3)}'', W_{\pi(1)}'')$,
      $M_2 \subseteq E_G(W_{\pi(3)}'', W_{\pi(2)}'')$ and
    \item
      $M_3$ consists of three edges, one from 
      each of $E(W_i'', W_{i+1}'')$ for $i \in [3]$. 
  \end{itemize}
  Let $M' = M_1 \cup M_2$ and $M = M' \cup M_3$.
  Note that since $\s{W}''$ is $(8 \tau^2)$-equitable,
  $|M'| < |M| \le 2(8\tau^{2}n) + 3 \le 17 \tau^2 n$.
  Let $vv' \in M$.
  Since $v$ and $v'$ are both $\tau$-cyclic for $\s{W}$,
  \eqref{eq:in_out_good_to_U} and \eqref{eq:good_to_U} give
  us that the number of vertices $x \in U$ such that $xvv'$ is a transitive triangle is at least 
  \begin{equation*}
    |N^-_G(v, U) \cap N_G(v', U)|
    \ge n/6 - \alpha n - 2 \tau n.
  \end{equation*}
  Therefore, we can find the desired collection $\s{T}_3$ by 
  matching edges $vv'$ in either $M$ or $M'$ to unused vertices $x$ in $U'$
  such that $vv'u$ is a transitive triangle.
  We can clearly satisfy Property~\ref{prop:P4}.
  Note that Properties~\ref{prop:P1} and~\ref{prop:P4} imply that 
  $|V(G) \setminus Y_3| \in 9 \Z$.
  So we can satisfy Property~\ref{prop:P5}
  by picking $M$ or $M'$ appropriately.
\end{proof}

{\bf Acknowledgement.} We are thankful to Wojciech Samotij, who participated in many fruitful discussions in the beginning of the project.
We would also like to thank an anonymous referee for carefully reading
this manuscript and for providing many helpful comments.

This research was carried out whilst the second author was visiting the Department of Mathematics of the University of Illinois at Urbana--Champaign.
This author would like to thank the department for the hospitality he received.

The authors were also supported by the BRIDGE strategic alliance between the University of Birmingham and the University of Illinois at Urbana-Champaign, as part of the `Building Bridges in Mathematics' BRIDGE Seed Fund project.
The second author was supported by the University of Birmingham North America 
Travel Fund and the University of Birmingham Transatlantic Collaboration Fund.

%%%%%%%%%%%%%%%%%%%%%%%%%%%%%%%%%%%%%%%%%%%%%%%%%%%%%%%%%%%%%%%%%%%%%%%%%%%%%%%%%%%%%%%%%%%%%%%%%%%%%%%%%


\begin{thebibliography}{1}

 % \bibitem{MR1885388}
  %  N.~Alon and J.~H. Spencer, \emph{The probabilistic method}, second ed.,
  %  Wiley-Interscience Series in Discrete Mathematics and Optimization,
  %  Wiley-Interscience [John Wiley \& Sons], New York, 2000, With an appendix on
  %  the life and work of Paul Erd{\H{o}}s.

    \bibitem{corradi1963maximal} K.~Corr{\'a}di and A.~Hajnal, On
    the maximal number of independent circuits in a graph, 
    \emph{Acta Mathematica Hungarica} \textbf{14} (1963), no.~3, 423--439.

    \bibitem{cuckler}
    B.~Cuckler, On the number of short cycles in regular tournaments,
    unpublished manuscript (2008).

  \bibitem{cdkm2013}
    A.~Czygrinow, L.~DeBiasio, H.A.~Kierstead and T.~Molla, 
    An extension of the  {H}ajnal-{S}zemer\'edi theorem to directed graphs, 
    {\it Combin.\ Probab.\ Comput.} \textbf{24} (2015), no.~5, 754--773. 


  \bibitem{ckm2013}
    A.~Czygrinow, H.A.~Kierstead and T.~Molla, 
	On directed versions of the Corr\'adi-Hajnal Corollary,
        {\it European J.\ Combin.} \textbf{242} (2014), 1--14. 


  \bibitem{hajnal1970pcp} A.~Hajnal and E.~Szemer{\'e}di, Proof
  of a conjecture of {P}. {E}rd{\H{o}}s, \emph{Combinatorial Theory
  and Its Application} \textbf{2} (1970), 601--623.

  \bibitem{luczak} S. Janson, T. {\L}uczak and A. Ruci{\'n}ski, \emph{Random graphs}, 
    Wiley-Interscience Series in Discrete Mathematics and Optimization, Wiley-Interscience, New York, 2000. 


    \bibitem{keevash2009} P.~Keevash and B.~Sudakov,
    Triangle packings and $1$-factors in oriented graphs,
    {\it J.\ Combin.\ Theory Ser.~B} \textbf{99} (2009), no.~4, 709--727.

  \bibitem{KO-match}
    D.~K\"{u}hn and D.~Osthus, Multicolored Hamilton cycles and perfect
    matchings in pseudorandom graphs, {\it SIAM J.\ Discrete Math.}, Vol. 20, No. 2, (2006), 273--286.

  \bibitem{MR2500161}
    V.~R{\"o}dl, A.~Ruci{\'n}ski, and E.~Szemer{\'e}di, 
	Perfect matchings in large uniform hypergraphs with large minimum collective degree, 
	{\it J.\ Combin.\ Theory Ser.~A} \textbf{116} (2009), no.~3, 613--636.


    \bibitem{treglown_note}
    A.~Treglown, 
	A note on some embedding problems for oriented graphs,
    {\it J.\ of Graph Theory} \textbf{69} (2012), no.~3, 330--336.

    \bibitem{treglown_hsz}
    A.~Treglown, 
    On directed versions of the {H}ajnal-{S}zemer{\'e}di Theorem,	
    {\it Combin.\ Probab.\ Comput.} \textbf{24} (2015), no.~6, 873--928 

    \bibitem{wang00}
    H.~Wang, Independent directed triangles in a directed graph, 
    {\it Graphs Combin.} \textbf{16} (2000), no.~4, 453--462.

    \bibitem{yuster}
    R.~Yuster, Combinatorial and computation aspects of graph packing
    and graph decomposition, \it{Comput.\ Sci.\ Rev.} \textbf{1} (2007), no.~1, 
    12--26.

\end{thebibliography}
\end{document}